\documentclass[10pt]{amsart}

\usepackage[T1]{fontenc}
\usepackage[utf8]{inputenc}
\usepackage[english]{babel}

\usepackage{amssymb,bm,fix-cm,graphicx,mathtools,parskip,thmtools}
\usepackage{newtxtext,newtxmath}
\usepackage[scr=boondox,cal=boondoxo,bb=boondox,frak=euler]{mathalfa}
\usepackage{microtype}
\usepackage{csquotes,amsfonts,amsmath,amsthm}
\usepackage{longtable}
\usepackage{hyperref}
\usepackage{cleveref}

\allowdisplaybreaks[4]

\def\epsilon{\varepsilon}

\def\ZZ{\mathbb{Z}}
\def\QQ{\mathbb{Q}}
\def\RR{\mathbb{R}}
\def\CC{\mathbb{C}}

\def\cA{\mathcal{A}}

\def\cC{\mathcal{C}}

\def\cJ{\mathcal{J}}
\def\cK{\mathcal{K}}

\def\cQ{\mathcal{Q}}

\def\cS{\mathcal{S}}
\def\cT{\mathcal{T}}

\def\cX{\mathcal{X}}
\def\cY{\mathcal{Y}}

\def\SS{\texorpdfstring{$\mathcal{S}$}{}}
\def\btau{\bm{\tau}}

\newcommand{\hyp}{\mathrm{hyp}}
\newcommand{\RV}{\mathrm{RV}}
\newcommand{\Rtop}{\mathrm{t}}
\newcommand{\Rbot}{\mathrm{b}}
\newcommand{\Leb}{\mathrm{Leb}}
\newcommand{\Rec}{\mathrm{Rec}}
\newcommand{\pr}{\mathfrak{p}}
\newcommand{\Id}{\mathrm{Id}}
\newcommand{\Es}{\mathrm{E}^\mathrm{s}}
\newcommand{\tr}{{\mbox{\raisebox{0.33ex}{\scalebox{0.6}{$\intercal$}}}}}
\newcommand{\ort}{{\mbox{\raisebox{0.33ex}{\scalebox{0.6}{$\perp$}}}}}
\newcommand{\diff}{\mathrm{d}}
\newcommand{\R}{\mathrm{R}}
\newcommand{\D}{\mathrm{D}}
\newcommand{\col}{\mathrm{col}}

\usepackage[style=alphabetic,natbib=true,backend=biber,giveninits=true,minalphanames=3,maxalphanames=4,maxbibnames=99]{biblatex}
\newtheorem{thm}{Theorem}[section]
\newtheorem{conv}{Convention}
\newtheorem{lem}[thm]{Lemma}
\newtheorem{prop}[thm]{Proposition}
\newtheorem{cor}[thm]{Corollary}

\theoremstyle{definition}

\theoremstyle{remark}

\begin{document}

	\title{Topological weak mixing for certain linear involutions of hyperelliptic type}
	
	\author{Felipe Arbulú}
	\address{
	Laboratoire Amiénois de Mathématique Fondamentale et Apliquée,
	CNRS-UMR 7352,
	Université de Picardie Jules Verne,
	33 rue Saint Leu,
	Amiens,
	France.}
	\email{farbulu@u-picardie.fr}

	\author{Bastián Espinoza}
	\address{
	Département de Mathématiques,
	Université de Liège,
	Allée de la Découverte 12 (B37),
	B-4000 Liège,
	Belgium.}
	\email{baespinoza@uliege.be}

	\author{Alejandro Maass}
	\address{Departamento de Ingeniería Matemática,
	Centro de Modelamiento Matemático and Millennium Institute Center for Genome Regulation,
	Santiago,
	Chile,
	Universidad de Chile and IRL-CNRS 2807,
	Beauchef 851,
	Santiago, Chile.}
	\email{amaass@dim.uchile.cl}

	\begin{abstract}
		A linear involution is an injective piecewise isometry defined on a pair of disjoint intervals.
		They are defined by a combinatorial data given by a generalized permutation and a length vector.
		As it was done for interval exchange transformations, it is conjectured that, except for some combinatorial data, a typical linear involution is measure-theoretically weakly mixing.
		In this direction, one may first explore the question of topological weak mixing for topological models of linear involutions.
		In this article we prove topological weak mixing for the natural symbolic codings of typical linear involutions defined by some generalized permutations of hyperelliptic type.
		We consider the generalized permutation
		{\footnotesize
			\[
				\sigma_{s,r}
				=
				\begin{pmatrix}
					0 & A & 1 & 2 & \cdots & s & A & s+1 & s+2 & \cdots & s+r \\
					s+r & \cdots & s+2 & s+1 & B & s & \cdots & 2 & 1 & B & 0
				\end{pmatrix}.
			\]
		}
		We prove that the natural symbolic coding of a typical linear involution defined by a generalized permutation in the Rauzy class of $\sigma_{s,r}$ is topologically weakly mixing, provided that it has at least a simple letter and its associated genus is sufficiently large.
	\end{abstract}

	\maketitle
	\markboth{F. ARBULÚ, B. ESPINOZA AND A. MAASS}{TOPOLOGICAL WEAK MIXING FOR CERTAIN LINEAR INVOLUTIONS OF HYPERELLIPTIC TYPE}

	\section{Introduction}

	Interval exchange transformations form an important class of dynamical systems of geometric origin exhibiting a rich variety of ergodic and spectral behaviors \cites{Via06, Yoc10, Zor06, Kea75, Kat80, Mas82, Vee82, AF07}.
	They arise as Poincaré maps of translation flows on translation surfaces.
	It follows from the work of Katok that interval exchange transformations are not mixing in the measure-theoretic sense \cite{Kat80}.
	The question of whether a typical interval exchange transformation which is not of rotation type is measure-theoretically weakly mixing remained open until the celebrated work of Avila and Forni \cite{AF07}, which settled this conjecture.
	Since then, measure-theoretic weak mixing for typical directional flows on translation surfaces has been studied.
	We refer to \cite{AD16} for the question of measure-theoretic weak mixing for almost every directional flow on non-arithmetic Veech surfaces, and to \cite{AHCF24} for recent results on measure-theoretic weak mixing for polygonal billiards.
    Before measure-theoretic weak mixing was established for typical interval exchange transformations, Nogueira and Rudolph proved topological weak mixing in \cite{NR97}, later extended to translation flows by Lucien in \cite{Luc98}.
    These works introduced foundational techniques that were subsequently adapted to the measure-theoretic setting.

	Linear involutions were introduced by Danthony and Nogueira in \cites{DN88, DN90} as a natural generalization of interval exchange transformations.
	In our setting, these transformations arise as Poincaré maps induced on two copies of an interval by a non-orientable measured foliation on an orientable surface.
	As in the case of interval exchange transformations, a typical linear involution was shown to be minimal and uniquely ergodic \cite{DN88}.
	However, for linear involutions the question of weak mixing (in both senses considered above) is less clear.
	As with interval exchange transformations, linear involutions have discontinuities in their domains.
	We therefore consider topological models of minimal linear involutions that correspond to the natural symbolic codings of their orbits with respect to the partition given by the intervals defining the involution.
	For a typical linear involution, both the involution and its natural symbolic coding are uniquely ergodic, and the corresponding measure-preserving dynamical systems are isomorphic.
	Then both measure theoretic and topological weak mixing can be studied through the (measurable or continuous) eigenvalues of the system: these notions correspond to the absence of non-trivial eigenvalues.
	Several necessary and sufficient conditions for a complex number to be a continuous or a measurable eigenvalue of a topological or measurable dynamical system defined on a Cantor space have been proposed in the literature; see, for instance, \cites{DFM19, BCE26}.

	In this article, we investigate the question of topological weak mixing for the natural symbolic codings of typical linear involutions.
	A first step is to provide $\cS$-adic representations of these natural symbolic codings.
	We describe such $\cS$-adic representations in \Cref{thm:RV_adic}, where the associated directive sequences are obtained from the Rauzy--Veech induction algorithm for linear involutions.
	A key technical point is that the resulting morphisms are not proper, in contrast with the classical setting of interval exchange transformations.
	Recent advances in the study of continuous eigenvalues of $\cS$-adic subshifts of this type apply to our setting \cite{BCE26}.
	In particular, they show that the natural symbolic coding of a typical linear involution possesses a non-trivial letter-coboundary.

	We study linear involutions defined by certain generalized permutations of hyperelliptic type.
	These permutations are natural objects, as they provide representatives of certain strata of meromorphic quadratic differentials on Riemann surfaces \cites{Lan04, Lan08, Zor08}; see the table below.
	Moreover, the half-translation surfaces associated with these permutations possess a central symmetry, since they arise as branched double covers of the Riemann sphere.
	As in the case of interval exchange transformations, it is conjectured that a typical linear involution is measure-theoretically weakly mixing, provided that a suitable geometric quantity is sufficiently large.
	Some obstructions to measure-theoretically weak mixing arise naturally for linear involutions: if every letter is duplicated in the generalized permutation, then $-1$ is a rational eigenvalue of the system; and, in the hyperelliptic setting, if $s = r = 0$, then the resulting measure-preserving dynamical system is isomorphic to a circle rotation.

	\begin{longtable}[ht!]{|c|c|c|c|}
			\hline
			$s$ & $r$ & Connected component & Genus
			\\
			\hline
			$2k+1$ & $2j+1$ & $\cQ(4j+2, 4k+2)^\hyp$ & $k+j+2$
			\\
			\hline
			$2k+1$ & $2j$ & $\cQ(2j-1,2j-1,4k+2)^\hyp$ & $k+j+1$
			\\
			\hline
			$2k$ & $2j+1$ & $\cQ(4j+2,2k-1,2k-1)^\hyp$ & $k+j+1$
			\\
			\hline
			$2k$ & $2j$ & $\cQ(2j-1,2j-1,2k-1,2k-1)^\hyp$ & $k+j$
			\\
			\hline
		\noalign{\vskip 0.5em}
		\caption{Some representatives of hyperelliptic components of the strata of meromorphic quadratic differentials defined on Riemann surfaces.}
		\label{table:connected_components}
	\end{longtable}

	Our main result is the following:
	
	\smallbreak
	\begin{thm} \label{thm:main}
		Consider a generalized permutation $\pi$ in the Rauzy class of $\sigma_{s,r}$ that has at least a simple letter and let $g$ be its associated genus.
		Suppose that $g \geq 3$ if $s$ and $r$ are even and $g \geq 5$ in any other case.
		Then for almost all admissible length vector $\lambda$ the natural symbolic coding of a linear involution defined by $(\pi, \lambda)$ is topologically weakly mixing.
	\end{thm}

    We remark that topological weak mixing of the natural symbolic coding implies topological weak mixing, in the sense of Nogueira and Rudolph \cite[Theorem 6.5]{NR97}, of the corresponding linear involution; see \cite[Section 6]{NR97} for further details.

	To prove \Cref{thm:main}, we first provide an $\cS$-adic representation of the natural symbolic coding of a linear involution that has the Keane property.
	The directive sequence defining this natural symbolic coding is obtained from the Rauzy--Veech induction algorithm; in particular, we make use of the fact that the resulting incidence matrices are symplectic.
	We then apply a recent necessary condition for a complex number to be a continuous eigenvalue of the constructed symbolic coding \cite{BCE26}.
	The main difference from the case of interval exchange transformations lies in the existence of a non-trivial letter-coboundary, which contributes with an additional dimension to the stable space appearing in the well-known necessary condition in that setting.
	Finally, we use a Nogueira--Rudolph type argument \cite{NR97} to rule out the remaining locus of possible continuous eigenvalues for hyperelliptic generalized permutations.

	The article is organized as follows.
	In \Cref{sec:linear_involutions} we review the basic definitions and background of linear involutions.
	In \Cref{sec:natural_codings} we provide $\cS$-adic representations of the natural symbolic codings of linear involutions.
	In \Cref{sec:continuous_eigenvalues} we give a necessary condition for a complex number to be a continuous eigenvalue of such natural symbolic codings.
	Finally, in \Cref{sec:proof} we prove \Cref{thm:main}.

	\section{Linear involutions} \label{sec:linear_involutions}

	\subsection{Generalized permutations} In this section we define linear involutions and we recall the Rauzy--Veech induction algorithm on linear involutions by Danthony and Nogueira \cite{DN90} and combinatorially for generalized permutations by Boissy and Lanneau \cite{BL09} (see also the work by Avila and Resende \cite{AR12}).

	Let $\cA$ be a finite set of cardinality $d \geq 2$ and $\ell, m$ be positive integers satisfying $\ell + m = 2d$.
	A \emph{generalized permutation} of type $(\ell, m)$ is a two-to-one map $\pi \colon \{1, \dotsc, 2d\} \to \cA$.
	We usually represent such a map by a table with two rows 

	{\small
		\[
			\pi =
			\begin{pmatrix}
				\pi(1) & \pi(2) & \cdots & \pi(\ell) \\
				\pi(\ell+1) & \pi(\ell+2) & \cdots & \pi(\ell+m)
			\end{pmatrix}.
		\]
	}
	An involution $\sigma \colon \{1, \dotsc, 2d\} \to \{1, \dotsc, 2d\}$ is defined naturally from a generalized permutation by the rules $\sigma(i) \neq i$ and $\pi(\sigma(i)) = \pi(i)$ for every $i \in \{1, \dotsc, 2d\}$.
	That is, $\{i, \sigma(i)\}$ are the two preimages of the letter $\pi(i) = \pi(\sigma(i))$ by $\pi$.
	We say that a letter is \emph{duplicate} if both of its occurrences are in the same row of previous representation of $\pi$ and \emph{simple} otherwise.
	We will also assume the following convention \cite[Convention 2.7]{BL09}:
	\smallbreak
	\begin{conv}\label{conv:duplicate}
		Every generalized permutation contains duplicate letters in both rows.
	\end{conv}

	\smallbreak
	We now recall the definition of an irreducible generalized permutation in the sense of \cite[Definition 3.1]{BL09}.
	A \emph{decomposition} of a generalized permutation $\pi$ is a way of writing it as
	\[
		\pi = \left(\begin{array}{c|c|c}
			F_{\mathrm{tl}} & *** & F_{\mathrm{tr}} \\\hline
			F_{\mathrm{bl}} & *** & F_{\mathrm{br}}
		\end{array}\right)
	\]
	where $F_{\mathrm{tl}}, F_{\mathrm{tr}}, F_{\mathrm{bl}}, F_{\mathrm{br}}$ are (possibly empty) subsets of $\cA$.
	This notation means that there exist $1 \leq i_1 \leq i_2 \leq \ell < i_3 \leq i_4 \leq \ell + m$ such that
	\begin{itemize}
		\item $F_{\mathrm{tl}} = \{\pi(1), \dotsc, \pi(i_1)\}$;
		\item $F_{\mathrm{tr}} = \{\pi(i_2), \dotsc, \pi(\ell)\}$;
		\item $F_{\mathrm{bl}} = \{\pi(\ell+1), \dotsc, \pi(i_3)\}$;
		\item $F_{\mathrm{br}} = \{\pi(i_4), \dotsc, \pi(\ell+m)\}$.
	\end{itemize}
	We de not assume that $|F_{\mathrm{tl}}| = |F_{\mathrm{bl}}|$ or $|F_{\mathrm{tr}}| = |F_{\mathrm{br}}|$.
	Once a decomposition is clear from context, we refer to $F_{\mathrm{tl}}, F_{\mathrm{tr}}, F_{\mathrm{bl}}, F_{\mathrm{br}}$ as the top-left, top-right, bottom-left and bottom-right corners of $\pi$, respectively.
	
	Let $\pi$ be a strict generalized permutation. We say that $\pi$ is \emph{reducible} if there exists a decomposition
	\[
		\pi = \left(\begin{array}{c|c|c}
			A \cup B & *** & D \cup B \\\hline
			A \cup C & *** & D \cup C 
		\end{array}\right)
	\]
	where $A,B,C,D$ are disjoint (possibly empty) subsets of $\cA$ satisfying one of the following conditions:
	\begin{itemize}
		\item no corner is empty;
		\item there is exactly one empty corner and it is on the left;
		\item there are exactly two empty corners and they are on the same side.
	\end{itemize}
	Otherwise, we say that it is \emph{irreducible}.
	A generalized permutation stems as a measured foliation of a quadratic differential on a Riemann surface or, equivalently, admits a suspension datum if and only if it is irreducible and satisfies \Cref{conv:duplicate} \cite[Theorem 3.2]{BL09}.
	Furthermore, to construct an $\cS$-adic representation of the natural symbolic coding of a linear involution, it is necessary for the later to have the Keane property (see the definition below), which is true for almost all admissible parameters assuming that the generalized permutation is irreducible \cite[Theorem B]{BL09}. 
	From now on, we will always assume that a generalized permutation is irreducible and satisfies \Cref{conv:duplicate}.

	We consider the alphabet $\cA \cup \cA^{-1}$ by adding inverse letters, where $(\alpha^{-1})^{-1} = \alpha$ for every $\alpha \in \cA$.
	In this way, a generalized permutation $\pi$ naturally defines a \emph{labelled} generalized permutation, which is a bijection $\flat(\pi) \colon \{1, \dotsc, 2d\} \to \cA \cup \cA^{-1}$ defined as follows: for each $i \in \{1, \dotsc, 2d\}$, either $\flat(\pi)(i) = \pi(i)$ and $\flat(\pi)(\sigma(i)) = \pi(i)^{-1}$ or $\flat(\pi)(i) = \pi(i)^{-1}$ and $\flat(\pi)(\sigma(i)) = \pi(i)$.
	It is clear that in this way that labelled and unlabelled generalized permutations are in finite-to-one correspondence.
	We also represent labelled generalized permutations as a table with two rows and we say that a letter $\alpha \in \cA \cup \cA^{-1}$ is duplicate if both $\alpha$ and $\alpha^{-1}$ occur in the same row, and simple otherwise.

	\subsection{Linear involutions} Consider two copies $I \times \{0\}$ and $I\times \{1\}$ of an open interval $I$ of the real line, which are the \emph{components} of $\tilde{I} = I \times \{0, 1\}$.
	For an unlabelled generalized permutation $\pi$ of type $(\ell, m)$, we can associate an \emph{admissible length (column) vector} $\lambda \in \RR^{\cA}_+$ such that
	\[
		\lambda_{\pi(1)} + \lambda_{\pi(2)} + \dotsb + \lambda_{\pi(\ell)}
		= \lambda_{\pi(\ell+1)}+ \lambda_{\pi(\ell+2)} + \dotsb + \lambda_{\pi(\ell+m)}.
	\]
	Define $\Delta_\pi$ to be the set of all \emph{normalized} admissible length vectors relative to the generalized permutation $\pi$, that is, the set of all admissible length vectors $\lambda \in \RR^{\cA}_+$ such that $\sum_{\alpha \in \cA} \lambda_\alpha = 1$.
	This set is an hyperplane section of the simplex of probability vectors and carries a natural measure $\Leb_\pi$.
	We consider $\flat(\pi)$ a labelled version of $\pi$ and a partition of $I \times \{0\}$ (minus $\ell-1$ points) into $\ell$ open intervals $(I_{\flat(\pi)(i)} \colon 1 \leq i \leq \ell)$ of respective lengths $\lambda_{\pi(1)}, \lambda_{\pi(2)}, \dotsc, \lambda_{\pi(\ell)}$ and a partition of $I \times \{1\}$ (minus $m-1$ points) into $m$ open intervals $(I_{\flat(\pi)(\ell+i)} \colon 1 \leq i \leq m)$ of respective lengths $\lambda_{\pi(\ell+1)}, \lambda_{\pi(\ell+2)}, \dotsc, \lambda_{\pi(\ell+m)}$.
	Denote by $\Sigma$ the $2d-2$ points separating the intervals $(I_\alpha \colon \alpha \in \cA \cup \cA^{-1})$.
	The \emph{linear involution} on $I$ defined by the data $(\pi, \lambda)$ relative to the alphabet $\cA$ is the map $T$ defined on $\tilde{I} \setminus \Sigma$ given by $T = \sigma_2 \circ \sigma_1$, where $\sigma_1, \sigma_2$ are involutions defined on $\tilde{I} \setminus \Sigma$ as follows:

	\begin{itemize}
		\item if $\alpha \in \cA$ is duplicate, then the restriction of $\sigma_1$ to $I_\alpha$ is a symmetry from $I_\alpha$ to $I_{\alpha^{-1}}$ and if $\alpha$ is simple, then the restriction of $\sigma_1$ to $I_\alpha$ is a translation from $I_\alpha$ to $I_{\alpha^{-1}}$;

		\item the involution $\sigma_2$ exchanges the components of $\tilde{I}$, i.e., sends $(x, \epsilon)$ to $(x, 1-\epsilon)$, for every $x \in \tilde{I} \setminus \Sigma$ and $\epsilon \in \{0,1\}$.
	\end{itemize}

	By a linear involution defined by the data $(\pi, \lambda)$ we mean a linear involution on the interval $I = (0, \sum_{i = 1}^\ell \lambda_{\pi(i)})$ relative to the alphabet $\cA$ which defines the generalized permutation $\pi$.
	Our definition of linear involution coincides with coherent and non orientable linear involutions in the terminology of \cite[Section 3]{BDD+17}: these correspond to non-orientable laminations on orientable surfaces.
	\Cref{fig:LI} illustrates an example of a linear involution.

	\begin{figure}[ht]
	    \centering
	    \begin{minipage}{0.48\textwidth}
	        \centering
	        \includegraphics[width=\textwidth]{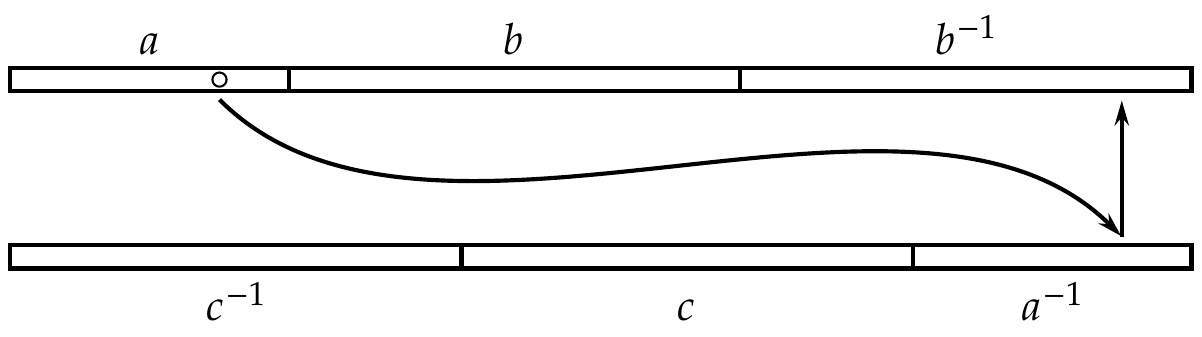}
	    \end{minipage}
	    \hfill
	    \begin{minipage}{0.48\textwidth}
	        \centering
	        \includegraphics[width=\textwidth]{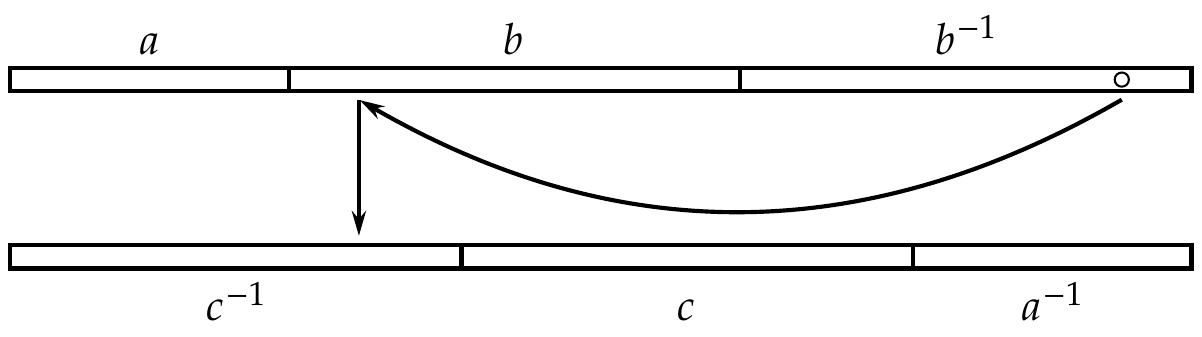}
	    \end{minipage}

	    \caption{The first two iterations of a point under a linear involution on the interval $I = (0,1)$ defined by the data $(\pi, \lambda)$ relative to the alphabet $\cA=\{a, b , c\}$.
	    Here $\pi = \Big( \begin{smallmatrix} a & b & b \\ c & c & a \end{smallmatrix} \Big)$ and $(\lambda_a, \lambda_b, \lambda_c) = \Big( \sqrt{5} - 2, \frac{3 - \sqrt{5}}{2}, \frac{3 - \sqrt{5}}{2} \Big)$.
	    The associated labelled permutation is $\flat(\pi) = \Big( \begin{smallmatrix} a & b & b^{-1} \\ c^{-1} & c & a^{-1} \end{smallmatrix} \Big)$. 
        }
	    \label{fig:LI}
	\end{figure}

	A \emph{connection} of a linear involution $T$ is a triple $(x, y, n)$, where $x$ belongs to $\sigma_2(\Sigma)$, $y$ belongs to $\Sigma$ and $n \ge 0$ is such that $T^n x = y$.
	We say that $T$ has the \emph{Keane property} if $T$ has no connections.
	Denote by $O$ the union of orbits of points in $\Sigma$ and $\tilde{O} = O \cup \sigma_2(O)$.
	Then $T$ is a bijection from $\tilde{I} \setminus \tilde{O}$ onto itself, $\tilde{I} \setminus \tilde{O}$ is dense in $\tilde{I}$ and the orbit of a point of $\tilde{I} \setminus \tilde{O}$ is well defined.
	Moreover, if $T$ has the Keane property, such a non negative orbit is dense in $\tilde{I}$ \cite[Proposition 4.2]{BL09}.
	In addition, if $\cK_\pi$ is the set of $\lambda \in \Delta_\pi$ such that a linear involution defined by the data $(\pi, \lambda)$ has the Keane property, then $\cK_\pi$ is measurable in $\Delta_\pi$ by \cite[Proposition 4.2]{BL09} and from \cite[Theorem B]{BL09} we have that $\Leb_\pi(\Delta_\pi \setminus \cK_\pi) = 0$.

	\subsection{Rauzy--Veech induction} We now define the Rauzy--Veech induction algorithm for unlabelled generalized permutations as defined by Boissy and Lanneau \cite[Section 2.2]{BL09}.
	Analogous definitions can be given for labelled generalized permutations.
    It consists of at most two operations $R_{\Rtop}$ and $R_{\Rbot}$ on an unlabelled generalized permutation $\pi$, which we call \emph{top} and \emph{bottom}.
	
	If $\sigma(\ell) > \ell$, then $R_{\Rtop}(\pi)$ is the type-$(\ell, m)$ generalized permutation defined as:
	\[
		R_{\Rtop}(\pi)(i)
		= \begin{cases}
			\pi(i) & i \leq \sigma(\ell) \\
			\pi(\ell+m) & i = \sigma(\ell) + 1 \\
			\pi(i-1) & \text{otherwise};
		\end{cases}
	\]
	if $\sigma(\ell) < \ell$ and there exists a duplicate letter in the bottom row of $\pi$ which is not the last letter, then $R_{\Rtop}(\pi)$ is the type-$(\ell+1, m-1)$ generalized permutation defined as:
	\[
		R_{\Rtop}(\pi)(i)
		= \begin{cases}
			\pi(i) & i < \sigma(\ell) \\
			\pi(\ell+m) & i = \sigma(\ell) \\
			\pi(i-1) & \text{otherwise};
		\end{cases}
	\]
	and, in any other case, $R_{\Rtop}$ is not defined on $\pi$.
	When a top operation is defined, we call $\pi(\ell)$ the \emph{winner} and $\pi(\ell+m)$ the \emph{loser} of the operation.
	
	Similarly, if $\sigma(\ell+m) < \ell$, then $R_{\Rbot}(\pi)$ is the type-$(\ell, m)$ generalized permutation defined as:
	\[
		R_{\Rbot}(\pi)(i)
		= \begin{cases}
			\pi(\ell) & i = \sigma(\ell+m) + 1 \\
			\pi(i-1) & \sigma(\ell+m) + 1 < i \leq \ell \\
			\pi(i) & \text{otherwise};
		\end{cases}
	\]
	if $\sigma(\ell+m) > \ell$ and there exists a duplicate letter in the top row of $\pi$ which is not the last letter, then $R_{\Rbot}(\pi)$ is the type-$(\ell-1, m+1)$ generalized permutation defined as:
	\[
		R_{\Rbot}(\pi)(i)
		= \begin{cases}
			\pi(i+1) & \ell \leq i < \sigma(\ell+m) + 1 \\
			\pi(\ell) & i = \sigma(\ell+m) - 1 \\
			\pi(i) & \text{otherwise};
		\end{cases}
	\]
	and, in any other case, $R_{\Rbot}$ is not defined on $\pi$.
	When a bottom operation is defined, we call $\pi(\ell+m)$ the \emph{winner} and $\pi(\ell)$ the \emph{loser} of the operation.
	Observe that, if $\pi$ is irreducible, then at least one of these operations is defined on $\pi$.
	Moreover, $R_{\Rtop}(\pi)$ and $R_{\Rbot}(\pi)$ are also irreducible if they are defined.
	We remark that for generalized permutations the winner and loser letters are different.
	For (labelled or unlabelled) generalized permutations, consider the directed graph whose vertices are the irreducible generalized permutations and such that $\pi \to \pi'$ is an edge if $R_{\Rtop}(\pi) = \pi'$ or $R_{\Rbot}(\pi) = \pi'$.
	The strongly connected components of this graph are called \emph{Rauzy classes} and the graph is called \emph{Rauzy diagram}.
	Rauzy classes are in a finite-to-one correspondence with the connected components of the strata of the moduli space of quadratic differentials \cite[Theorem D]{BL09}.

	From \cite[Section 2.2]{BL09} the dynamical interpretation of the Rauzy--Veech induction for linear involutions is given by the following proposition:

	\begin{prop} \label{prop:RV_induction}
		Let $T$ be a linear involution on $I$ defined by the data $(\pi, \lambda)$ relative to the alphabet $\cA$ and let $\flat(\pi)$ be a labelled version of $\pi$.
		\begin{itemize}
			\item If $\lambda_{\pi(\ell)} > \lambda_{\pi(\ell+m)}$ then the first return map $T'$ of $T$ on $(I \setminus I_{\flat(\pi)(\ell+m)}) \times \{0,1\}$ is a linear involution on $I \setminus I_{\flat(\pi)(\ell+m)}$ defined by the data $(\pi', \lambda')$ relative to the alphabet $\cA$, where $\pi' = R_{\Rtop}(\pi)$, $\lambda'_{\pi(\ell)} = \lambda_{\pi(\ell)} - \lambda_{\pi(\ell+m)}$ and $\lambda'_\alpha = \lambda_\alpha$ if $\alpha \neq \pi(\ell)$.
			\item If $\lambda_{\pi(\ell+m)} > \lambda_{\pi(\ell)}$ then the first return map $T'$ of $T$ on $(I \setminus I_{\flat(\pi)(\ell)}) \times \{0,1\}$ is a linear involution on $I \setminus I_{\flat(\pi)(\ell)}$ defined by the data $(\pi', \lambda')$ relative to the alphabet $\cA$, where $\pi' = R_{\Rbot}(\pi)$, $\lambda'_{\pi(\ell+m)} = \lambda_{\pi(\ell+m)} - \lambda_{\pi(\ell)}$ and $\lambda'_\alpha = \lambda_\alpha$ if $\alpha \neq \pi(\ell+m)$.
		\end{itemize}
		In this case we denote $\RV(\pi, \lambda) = (\pi', \lambda')$.
	\end{prop}

	We remark here that if a linear involution has the Keane property, we can always iterate the Rauzy--Veech induction \cite[Proposition 4.2]{BL09}, i.e., the iterates $\RV^n(\pi, \lambda) = (\pi^n, \lambda^n)$ are well-defined for every $n \ge 0$.
	In such case we denote by $\RV^n(T)$ the linear involution on $I^n$ defined by the data $\RV^n(\pi, \lambda)$ relative to the alphabet $\cA$.

	For an edge $\gamma = \pi \to \pi'$ in the unlabelled Rauzy diagram we define $\flat(\gamma) = \flat(\pi) \to \flat(\pi')$ to be the edge in the labelled Rauzy diagram joining the labelled versions of $\pi$ and $\pi'$ that is obtained by following the same top or bottom operation as in the edge $\gamma$ (here the labels on $\flat(\pi')$ are obtained from the labels on $\flat(\pi)$ following the Rauzy--Veech operation).
	We extend this definition to walks: for a directed walk $\gamma = \gamma_0 \gamma_1 \dotsb \gamma_{n-1}$ in the unlabelled Rauzy diagram we define $\flat(\gamma) = \flat(\gamma_0) \flat(\gamma_1) \dotsb \flat(\gamma_{n-1})$ to be the directed walk in the labelled Rauzy diagram obtained by following the same top or bottom operations for every edge as in the directed walk $\gamma$.
	Thus the Rauzy--Veech induction algorithm of a linear involution that has the Keane property naturally defines an infinite directed walk $\gamma(\pi, \lambda) = \gamma_0(\pi, \lambda) \gamma_1(\pi, \lambda) \dotsb$ (resp. $\flat(\gamma(\pi, \lambda)) = \flat(\gamma_0(\pi, \lambda)) \flat(\gamma_1(\pi, \lambda)) \dotsb$) that starts at the vertex $\pi$ (resp. $\flat(\pi)$) in the unlabelled (resp. labelled) Rauzy diagram.
	In such case we also define $\gamma^n(\pi, \lambda) = \gamma_0(\pi, \lambda) \gamma_1(\pi, \lambda) \dotsb \gamma_{n-1}(\pi, \lambda)$ for each $n \geq 1$.

	\subsection{Symplectic matrices} We denote by $\Id$ the identity matrix and for $\alpha, \beta \in \cA$ by $E_{\alpha \beta}$ the matrix that has only one non-zero coefficient, equal to $1$, at position $(\alpha,\beta)$.
	For an edge $\gamma = \pi \to \pi'$ in the unlabelled Rauzy diagram with winner $\alpha_{\mathrm{w}}$ and loser $\alpha_{\mathrm{l}}$ we define
	\[
		M_{\gamma}
		= \Id + E_{\alpha_{\mathrm{w}} \alpha_{\mathrm{l}}}.
	\]
	We extend this definition to walks: for a directed walk $\gamma = \gamma_0 \gamma_1 \dotsb \gamma_{n-1}$ in the unlabelled Rauzy diagram we define $M_{\gamma} = M_{\gamma_0} M_{\gamma_1} \dotsb M_{\gamma_{n-1}}$.
	We define the \emph{stable space} of a linear involution defined by the data $(\pi, \lambda)$ that has the Keane property as the set of row vectors defined by:
	\[
		\Es(\pi, \lambda)
		= \Big\{v \in \RR^\cA \colon \lim\limits_{n \to +\infty} v M_{\gamma^n(\pi, \lambda)} = 0\Big\}.
	\]
	We will use some definitions of Avila and Resende's work \cite{AR12}, which uses a slightly different formalism for generalized permutations.
	Let $\cA$ be an alphabet with $d$ letters and consider a letter $\ast \notin \cA$.
	For an unlabelled generalized permutation $\pi$ of type $(\ell, m)$ with $\ell + m = 2d$ defined on the alphabet $\cA$ we associate a map $\tau \colon \{1, 2, \dotsc 2d+1\} \to \cA \cup \{\ast\}$ such that $\tau(m+1) = \ast$ that can be written as:
	\[
		\begin{pmatrix}
			\tau(1) & \tau(2) & \dotsc & \tau(2d + 1)
		\end{pmatrix}
		= \begin{pmatrix}
			\pi(\ell + m) & \dotsc & \pi(\ell + 1) & \ast & \pi(1) & \dotsc & \pi(\ell)
		\end{pmatrix}.
	\]
	For a letter $\alpha$ in $\cA$ we define $i_\alpha, j_\alpha \in \{1, 2,\dotsc, 2d\} \setminus \{m+1\}$ by $\tau^{-1}(\alpha) = i_\alpha = j_\alpha$ and $i_\alpha < j_\alpha$.
	Then we define the alternate form $\Omega_\pi$ indexed by $\cA$ as:
	\[	
		(\Omega_\pi)_{\alpha \beta}
		= \begin{cases}
		+2 & \text{$j_\alpha < i_\beta$} \\
		-2 & \text{$j_\beta < i_\alpha$} \\
		+1 & \text{$i_\alpha < i_\beta < j_\alpha < j_\beta$} \\
		-1 & \text{$i_\beta < i_\alpha < j_\beta < j_\alpha$} \\
		0  & \text{otherwise.}
		\end{cases}	
	\]
	Let $\gamma = \pi \to \pi'$ be an edge in the unlabelled Rauzy diagram.
	By means of a standard computation we have
	\[
		M_{\gamma}^\tr \Omega_\pi M_{\gamma}
		= \Omega_{\pi'}.
	\]
	This property extends to walks: if $\gamma$ is a directed walk that starts at the vertex $\pi$ and ends at the vertex $\pi'$ we have $M_{\gamma}^\tr \Omega_\pi M_{\gamma} = \Omega_{\pi'}$.
	Let $H(\pi) = \RR^\cA \Omega_\pi$ be the image of $\Omega_\pi$ (we recall again that it consists of row vectors).
	We define the \emph{symplectic form} $\omega_\pi \colon H(\pi) \times H(\pi) \to \RR$ by
	\[
		\omega_\pi(u \Omega_\pi, v \Omega_\pi)
		= - \langle u, v \Omega_\pi \rangle,
	\]
	where $\langle \cdot, \cdot \rangle$ denotes the standard inner product on $\RR^\cA$ defined for row vectors.
	That is, $\omega_\pi$ is a well-defined, non-degenerate alternate $2$--form which is also \emph{integrable} in the sense that there exists a constant $C(\pi) > 0$ such that
	\[
		|\omega_\pi(u, v)|
		\leq C(\pi) \| u \| \| v \|
	\]
	for all $u$ and $v$ in $H(\pi)$.
	Let $\gamma = \pi \to \pi'$ be an edge in the unlabelled Rauzy diagram.
	Since $M_{\gamma}^\tr \Omega_\pi M_{\gamma} = \Omega_{\pi'}$ it follows that $M_\gamma$ (acting on row vectors) maps $H(\pi)$ onto $H(\pi')$ and that
	\[
		\omega_{\pi'}(u M_\gamma, v M_\gamma)
		= \omega_\pi(u, v)	
	\]
	for each $u$ and $v$ in $H(\pi)$.
	The fact that $\omega_\pi$ is a symplectic form on $H(\pi)$ implies that there exists $g_\pi \geq 1$ such that $\dim H(\pi) = 2g_\pi$.
	Moreover, if $\gamma$ is a directed walk that joins $\pi$ and $\pi'$ in the unlabelled Rauzy diagram, then $M_\gamma$ is an isomorphism between $H(\pi)$ and $H(\pi')$.
	It follows that the quantity $g_\pi$ only depends on the unlabelled Rauzy diagram.

	\subsection{Lyapunov exponents} By a \emph{measure-preserving dynamical system} $(X, \mu, T)$ we mean a space $X$ together with a $\sigma$--algebra, a finite measure $\mu$ defined on such $\sigma$--algebra and $T \colon X \to X$ a measurable and measure-preserving transformation. 
	From here we follow the content and notation of \cite[Appendix D and Lemma 7.14]{BDG+24}.
	Let $\pi$ be a generalized permutation.
	The authors define two measure-preserving dynamical systems (notations are quite technical but are not important for our purpose): $(\partial^{-} P^{(1)}(\pi), \nu_P^{(1)}(\pi), \RV^\partial)$ and $(X_1(\pi), \phi_{X_1(\pi)}, \R^\partial)$, which we now briefly describe.
	In our terminology, the space $X_1(\pi)$ corresponds to $\Delta_\pi$ and the space $\partial P^{(1)}(\pi)$ consists of points $(\lambda, \tau)$ in $X_1(\pi) \times Y(\pi)$ for some cone $Y(\pi)$, which satisfy a certain condition.
	The maps $\RV^\partial$ and $\R^\partial$ correspond to the Rauzy--Veech induction on those spaces.
	In particular, $\R^\partial(\lambda) = \RV(\pi, \lambda)$ for all $\lambda \in X_1(\pi)$.
	The measures $\phi_{X_1(\pi)}$ and $\Leb_\pi$ are absolutely continuous with respect to each other and the projection onto the first coordinate $p \colon \partial^{-} P^{(1)}(\pi) \to X_1(\pi)$ satisfies $\phi_{X_1(\pi)} = p_\ast \nu_P^{(1)}(\pi)$.
	Then, a \emph{discrete symplectic cocycle} $\D^\R$ for $(\partial^{-} P^{(1)}(\pi), \nu_P^{(1)}(\pi), \RV^\partial)$ can be defined following \cite[Section 7.13]{BDG+24}. That is, we define a linear map $\D^\R(t, (\lambda, \tau))$ for all $t \in \ZZ$ and $(\lambda, \tau) \in \partial^{-} P^{(1)}(\pi)$.
	We remark that $\D^\R(t, (\lambda, \tau))$ acts on row vectors (that is, it is the transpose of the cocycle defined in \cite{BDG+24}).
	It turns out that for a fixed $t \in \ZZ$, the value $\D^\R(t, (\lambda, \tau))$ depends only on $\lambda$, and thus $\D^\R$ projects to a discrete symplectic cocycle $\bar{\D}^\R$ for $(X_1(\pi), \phi_{X_1}(\pi), \R^\partial)$ defined by $\bar{\D}^\R(t, \lambda) = \D^\R(t, (\lambda, \tau))$ for every $t \in \ZZ$, $\lambda \in X_1(\pi)$ and all $\tau \in Y(\pi)$.
	Moreover, if $\lambda \in X_1(\pi)$ then $\bar{\D}^\R(1, \lambda) = M_{\gamma^{n(\pi, \lambda)}(\pi, \lambda)}$ for some positive integer $n(\pi, \lambda)$, see \cite[Section 7.13]{BDG+24}.
	In particular, the cocycle preserves the symplectic form $\omega_\pi$.
	The cocycle $\D^\R$ is $\log$--integrable with respect to $\nu_P^{(1)}(\pi)$ \cite[Lemma 7.14]{BDG+24}, i.e., for all $t \in \ZZ$ the maps $(\lambda, \tau) \mapsto \log \| \D^\R(t, (\lambda, \tau)) \|$ and $(\lambda, \tau) \mapsto \log \| \D^\R(t, (\lambda, \tau))^{-1} \|$ are integrable with respect to $\nu_P^{(1)}(\pi)$.

	We now prove that the cocycle $\bar{\D}^\R$ is $\log$--integrable with respect to $\phi_{X_1(\pi)}$.
	First we state an elementary lemma whose proof is left to the reader.

	\begin{lem} \label{lem:pushforward}
	Let $(X, \cX, \mu)$ and $(Y, \cY, \nu)$ be two finite measure spaces and $p \colon X \to Y$ a measurable (with respect to $\cX$ and $\cY$) surjective map such that $\nu = p_\ast \mu$.
	Let $f \colon Y\to \RR$ be a measurable map.
	Then, if $f \circ p \colon X \to \RR$ is integrable with respect to $\mu$, then $f \colon Y \to \RR$ is integrable with respect to $\nu$ and
	\[
	\int_Y f(y) \diff \nu(y)
	= \int_X (f \circ p)(x) \diff \mu(x).
	\]
	\end{lem}

	\begin{cor} \label{cor:log-int}
		The discrete symplectic cocycle $\bar{\D}^\R$ is $\log$--integrable with respect to $\phi_{X_1(\pi)}$.
	\end{cor}

	\begin{proof}
		Let $t \in \ZZ$.
		We use \Cref{lem:pushforward} with $(\partial^{-} P^{(1)}(\pi), \nu_P^{(1)}(\pi), \RV^\partial)$, $(X_1(\pi), \phi_{X_1(\pi)}, \R^\partial)$, $p$ the projection onto the first coordinate and $f_\pm(\lambda) = \log \| \bar{\D}^\R(t, \lambda)^{\pm 1} \|$ to obtain
		\begin{align*}
			\int_{X_1(\pi)} \log \| \bar{\D}^\R(t, \lambda)^{\pm 1} \| \diff \phi_{X_1(\pi)}(\lambda)
			&= \int_{\partial^{-} P^{(1)}(\pi)} \log \| \bar{\D}^\R(t, p(\lambda, \tau))^{\pm 1} \| \diff \nu_P^{(1)}(\pi)(\lambda, \tau) \\
			&= \int_{\partial^{-} P^{(1)}(\pi)} \log \| \D^\R(t, (\lambda, \tau))^{\pm 1} \| \diff \nu_P^{(1)}(\pi)(\lambda, \tau) \\
			&< +\infty.
		\end{align*}
	\end{proof}

	From \Cref{cor:log-int}, the Oseledets' theorem \cites{Ose68}[Chapter 4]{BP13} implies that there exist $|\cA|$ Lyapunov exponents $\theta_1 \geq \theta_2 \geq \dotsb \geq \theta_{|\cA|}$ of the cocycle $\bar{\D}^\R(1, \cdot)$ with respect to the measure $\phi_{X_1(\pi)}$.
	Moreover, by symplecticity \cite{Via07} there exists $I \subseteq \{1, 2, \dotsc |\cA|\}$ with $|I| = g_\pi$ such that $\theta_i$ and $- \theta_i$ are Lyapunov exponents for each $i \in I$.
	Observe that $g_\pi$ coincides with the quantity defined in the previous subsection.   
    From this we deduce the following:

	\begin{lem} \label{lem:lyapunov}
		We have $|\{1 \leq i \leq \cA \colon \theta_i < 0 \}| \leq |\cA| - g_\pi$.
	\end{lem}

	For $\lambda \in X_1(\pi)$ the \emph{stable space} of the cocycle $\bar{\D}^\R$ can be defined as
	\[
		\Es_{\bar{\D}^\R}(\pi, \lambda)
		= \Big\{v \in \RR^\cA \colon \lim\limits_{n \to +\infty} v \bar{\D}^\R(n, \lambda) = 0\Big\}.
	\]
	Observe that $\Es(\pi, \lambda)$ defined above is a subspace of $\Es_{\bar{\D}^\R}(\pi, \lambda)$.
	Recall that $\Leb_\pi$ is absolutely continuous with respect to $\phi_{X_1(\pi)}$.
	We will use the following well-known fact (see for example \cite[Section 2.3]{AF07}):

	\begin{lem} \label{lem:contracting}
		For $\Leb_\pi$--almost all $\lambda$ we have that
		\[
			\dim \Es_{\bar{\D}^\R}(\pi, \lambda)
			= |\{1 \leq i \leq \cA \colon \theta_i < 0 \}|.
		\]
	\end{lem}

	\section{Natural symbolic codings of linear involutions and its \SS-adic representations} \label{sec:natural_codings}

	\subsection{Natural symbolic coding of a linear involution}
	Let $T$ be a linear involution on $I$ defined by $(\pi, \lambda)$ relative to the alphabet $\cA$ that has the Keane property.
	If $x \in \tilde{I} \setminus \tilde{O}$, we define the symbolic sequence $\omega = (\omega_n \colon n \in \ZZ)$ in $(\cA \cup \cA^{-1})^{\ZZ}$ by $\omega_n = \alpha$ if $T^n(x) \in I_\alpha$ for all $n \in \ZZ$.
	Let $\Omega_T \subseteq (\cA \cup \cA^{-1})^\ZZ$ be the closure of the set of symbolic sequences constructed in this way for every $x \in \tilde{I} \setminus \tilde{O}$.
	Clearly the symbolic sequence corresponding to $T^n(x)$ is $S^n(\omega)$, where $S \colon (\cA \cup \cA^{-1})^\ZZ \to (\cA \cup \cA^{-1})^\ZZ$ is the left shift map given by $S(\omega) = (\omega_{n+1} \colon n \in \ZZ)$.
	As $T$ has the Keane property, $(\Omega_T, S)$ is a minimal aperiodic subshift \cite[Proposition 5.6]{BDD+17}.
	We call $(\Omega_T, S)$ the \emph{natural symbolic coding} of the linear involution $T$.
	We remark here that the language of $\Omega_T$ is \emph{reduced}, that is, for every $\alpha \in \cA \cup \cA^{-1}$ the word $\alpha \alpha^{-1}$ does not belong to the language, see \cite[Section 5.1]{BDD+17}.

	\subsection{\SS-adic subshifts} We define $\cS$-adic subshifts given by directive sequences on the alphabet $\cA \cup \cA^{-1}$ and give some of the properties we will require here.
	In more general settings, directive sequences of morphisms are defined on sequences of alphabets, but a single alphabet is enough for our purpose.	
	For more complete references about $\cS$-adic subshifts you can see \cites{BD14}[Chapter 6.4]{DP22}{BSTY19, DDMP21}.

    Let $F_\cA$ be the free group on the alphabet $\cA$.
	For a word $w$ in $F_\cA$ we denote by $|w|$ its length.
	By a \emph{directive sequence} on the alphabet $\cA \cup \cA^{-1}$ we mean a sequence of morphisms $\btau = (\tau_n \colon n \geq 0)$ with $\tau_n \colon F_\cA \to F_\cA$ for each $n \geq 0$.
	A directive sequence defines a sequence of subshifts $(X_{\btau}^{(n)}, S)$ for all $n \geq 0$, where a bi-infinite sequence $(x_n \colon n \geq 0)$ belongs to $X_{\btau}^{(n)}$ if and only if for every $k, l \in \ZZ$ with $k \leq l$ the word $x_k x_{k+1} \dotsb x_l$ occurs in the word $\tau_n \circ \tau_{n + 1} \circ \dotsb \circ \tau_{N - 1}(\alpha)$ for some $\alpha \in \cA \cup \cA^{-1}$ and $N > n$.
	We define $X_{\btau} = X_{\btau}^{(0)}$.
	The \emph{incidence matrix} of a morphism $\tau \colon F_\cA \to F_\cA$ is the matrix whose coefficient at position $(\alpha,\beta)$ is the number of occurrences of the letter $\alpha$ in the word $\tau(\beta)$.
	The directive sequence $\btau$ is \emph{unimodular} if $\det(M_{\tau_n}) = \pm 1$ for all $n \geq 0$.
	The directive sequence $\btau$ is \emph{everywhere growing} if $|\tau_0 \circ \tau_1 \circ \dotsb \circ \tau_{n-1}(\alpha)|$ tends to infinity for every $\alpha \in \cA \cup \cA^{-1}$ and \emph{primitive} if for every $n \geq 0$ there exists $m > n$ such that for every $\alpha,\beta \in \cA \cup \cA^{-1}$ the letter $\beta$ occurs in $\tau_n \circ \tau_{n+1} \circ \dotsb \circ \tau_{m-1}(\alpha)$.
	Observe that if $\btau$ is primitive, then it is everywhere growing. 
	The directive sequence $\btau$ is \emph{recognizable} \cite{BSTY19} if for all $n \geq 0$ and any $y \in X_{\btau}$ there exists a unique couple $(k, x)$ with $x \in X_{\btau}^{(n)}$ and $0 \leq k < |\tau_0 \circ \tau_1 \circ \dotsb \circ \tau_{n-1}(x_0)|$ such that $y = S^k \tau_0 \circ \tau_1 \circ \dotsb \circ \tau_{n-1}(x)$. 
	From \cite[Lemma 3.2]{DDMP21} we have that $\btau$ is recognizable if and only if for every $n \geq 0$ there exists a positive integer $r_n$ such that if $x  = S^k \tau_0 \circ \tau_1 \circ \dotsb \circ \tau_{n-1}(y)$, $0 \leq k < |\tau_0 \circ \tau_1 \circ \dotsb \circ \tau_{n-1}(y_0)|$, $x' = S^{k'} \tau_0 \circ \tau_1 \circ \dotsb \circ \tau_{n-1}(y')$, $0 \leq k' < |\tau_0 \circ \tau_1 \circ \dotsb \circ \tau_{n-1}(y'_0)|$ and $x_{-r_n} x_{-r_n + 1} \dotsb x_{r_n} = x'_{-r_n} x'_{-r_n + 1} \dotsb x'_{r_n}$, then $k = k'$ and $y_0 = y_0'$.
	We call $r_n$ the \emph{recognizability constant} of the subshift $(X_{\btau}^{(n)}, S)$ for the morphism $\tau_0 \circ \tau_1 \circ \dotsb \circ \tau_{n-1}$.
	The following lemma is well-known, but we include a proof for completeness:

	\begin{lem} \label{lem:everywhere_growing}
		Let $\btau = (\tau_n \colon n \geq 0)$ be a directive sequence on the alphabet $\cA \cup \cA^{-1}$.
		Suppose that $\btau$ is recognizable and everywhere growing.
		Then $(X_{\btau}, S)$ is minimal if and only if $\btau$ is primitive.
	\end{lem}

	\begin{proof}
		It is a classical fact that $(X_{\btau}, S)$ is minimal whenever $\btau$ is primitive \cite[Proposition 6.4.5]{DP22}.
		Now assume that $X_{\btau}$ is minimal and fix $n \geq 0$.
		Let $r_n$ be the recognizability constant of the subshift $(X_{\btau}^{(n)}, S)$ for the morphism $\tau_0 \circ \tau_1 \circ \dotsb \circ \tau_{n-1}$.
		By minimality, there exists $k \geq 0$ such that every word in the language of $X_{\btau}$ of length at least $k$ contains an occurrence of every word in the language of $X_{\btau}$ of length $2r_n + 1$.
		Since $\btau$ is everywhere growing, we can find $m > n$ satisfying $|\tau_n \circ \tau_{n+1} \circ \dotsb \circ \tau_{m-1}(\alpha)| \geq k$ for all $\alpha \in \cA \cup \cA^{-1}$.
		Hence for every $\alpha \in \cA \cup \cA^{-1}$ the word $\tau_0 \circ \tau_1 \circ \dotsb \circ \tau_{m-1}(\alpha)$ has occurrences of every word in the language of $X_{\btau}$ of length $2r_n + 1$.
		Since $r_n$ is the recognizability constant of $(X_{\btau}^{(n)}, S)$ for the morphism $\tau_0 \circ \tau_1 \circ \dotsb \circ \tau_{n-1}$, for every $\alpha,\beta \in \cA \cup \cA^{-1}$ the word $\tau_n \circ \tau_{n+1} \circ \dotsb \circ \tau_{m-1}(\alpha)$ has an occurrence of $\beta$.
		It follows that $\btau$ is primitive.
	\end{proof}

	\subsection{\SS-adic representations of the natural symbolic codings of linear involutions} Our main result in this section is the following description:
	\begin{thm} \label{thm:RV_adic}
		Let $T$ be a linear involution on $I$ defined by the data $(\pi, \lambda)$ relative to the alphabet $\cA$ that has the Keane property.
		Let $(\alpha_{\mathrm{w_n}} \colon n \ge 0)$ and $(\alpha_{\mathrm{l_n}} \colon n \ge 0)$ be the sequences in $\cA \cup \cA^{-1}$ consisting of the winners and losers of the iterates of the Rauzy--Veech induction algorithm, respectively.
		Then $\Omega_T = X_{\btau_T}$, where $\btau_T = (\tau_n \colon n \geq 0)$ is defined by
		\[
			\tau_n(\alpha)
			= \begin{cases}
				\alpha_{\mathrm{w_n}}^{-1} \alpha_{\mathrm{l_n}} & \text{if $\alpha = \alpha_{\mathrm{l_n}}$} \\
				\alpha_{\mathrm{l_n}}^{-1} \alpha_{\mathrm{w_n}}  & \text{if $\alpha = \alpha_{\mathrm{l_n}}^{-1}$} \\
				\alpha & \text{otherwise.} \\
			\end{cases}
		\]
	\end{thm}

	\begin{proof}
		The linear involution $\RV^n(T)$ is well-defined on $I^n$ for each $n \ge 0$.
		Denote by $(I_\alpha^n \colon \alpha \in \cA \cup \cA^{-1})$ the partition of $\tilde{I}^n$.
		From \Cref{prop:RV_induction} and using induction we see that $\{T^k I_\alpha^n \setminus \tilde{O} \colon \alpha \in \cA \cup \cA^{-1}, 0 \leq k < |\tau_0 \circ \tau_1 \circ \dotsb \circ \tau_{n-1}(\alpha)|\}$ is a partition of $\tilde{I} \setminus \tilde{O}$ such that if $x \in I_\alpha^n \setminus \tilde{O}$ then the itinerary of $x, Tx, \dotsc T^{|\tau_0 \circ \tau_1 \circ \dotsb \circ \tau_{n-1}(\alpha)| - 1} x$ is the word $\tau_0 \circ \tau_1 \circ \dotsb \tau_{n-1}(\alpha)$.
		Hence, $X_{\btau_T} \subseteq \Omega_T$ and, since $X_{\btau_T}$ is closed and $S$-invariant, by minimality we have $X_{\btau_T} = \Omega_T$.
	\end{proof}

	Recall that we denote by $\Id$ the identity matrix and for $\alpha, \beta \in \cA \cup \cA^{-1}$ by $E_{\alpha \beta}$ the matrix that has only one non-zero coefficient, equal to $1$, at position $(\alpha,\beta)$.
	From \Cref{thm:RV_adic} we have that the incidence matrix of $\tau_n$ is
	\[
		M_{\tau_n}
		= \Id + E_{\alpha_{\mathrm{w_n}} \alpha_{\mathrm{l_n}}^{-1}} + E_{\alpha_{\mathrm{w_n}^{-1}} \alpha_{\mathrm{l_n}}}
		= (\Id + E_{\alpha_{\mathrm{w_n}} \alpha_{\mathrm{l_n}}^{-1}})(\Id + E_{\alpha_{\mathrm{w_n}^{-1}} \alpha_{\mathrm{l_n}}}).
	\]

	This motivates the following definition.
	For an edge $\flat(\gamma) = \flat(\pi) \to \flat(\pi')$ in the labelled Rauzy diagram with winner $\alpha_{\mathrm{w_n}}$ and loser $\alpha_{\mathrm{l_n}}$ we define
	\[
		M_{\flat(\gamma)}
		= \Id + E_{\alpha_{\mathrm{w_n}} \alpha_{\mathrm{l_n}}^{-1}} + E_{\alpha_{\mathrm{w_n}^{-1}} \alpha_{\mathrm{l_n}}}.
	\]
	We extend this definition to walks: for a directed walk $\flat(\gamma) = \flat(\gamma_0) \flat(\gamma_1) \dotsb \flat(\gamma_{n-1})$ in the labelled Rauzy diagram we define $M_{\flat(\gamma)} = M_{\flat(\gamma_0)} M_{\flat(\gamma_1)} \dotsb M_{\flat(\gamma_{n-1})}$.
	We say that a (row) vector $(v(\alpha) \colon \alpha \in \cA \cup \cA^{-1})$ is \emph{symmetric} if $v(\alpha) = v(\alpha^{-1})$ for every $\alpha \in \cA \cup \cA^{-1}$.
	In such case, let $(\pr(v)(\alpha) \colon \alpha \in \cA)$ be the row vector defined by $\pr(v)(\alpha) = v(\alpha)$ for every $\alpha \in \cA$.
	Observe that if $v$ is symmetric and $\gamma$ is an edge in the unlabelled Rauzy diagram, then $v M_{\flat(\gamma)}$ is also symmetric and $\pr(v M_{\flat(\gamma)}) = \pr(v) M_{\gamma}$.
	Hence, by induction we have $\pr(v M_{\flat(\gamma)}) = \pr(v) M_{\gamma}$ for every directed walk $\gamma$.
	Let $e$ denote the row vector whose entries are all equal to one, which is symmetric.
	From the previous observation, $e M_{\tau_0} M_{\tau_1} \dotsb M_{\tau_{n-1}}$ is symmetric, so
	\[
		|\tau_0 \circ \tau_1 \circ \dotsb \circ \tau_{n-1}(\alpha)|
		= |\tau_0 \circ \tau_1 \circ \dotsb \circ \tau_{n-1}(\alpha^{-1})|, \quad \text{for each $\alpha \in \cA \cup \cA^{-1}$ and $n \geq 0$.}
	\]

	Since $(X_{\btau_T}, S)$ is a minimal aperiodic subshift and the directive sequence $\btau_T$ is unimodular, we have that $\btau_T$ is also recognizable \cite[Theorem 4.6]{BSTY19}.
	Moreover, the directive sequence $\btau_T$ is primitive:

	\begin{lem} \label{lem:primitive}
		If $T$ has the Keane property then the directive sequence $\btau_T$ is primitive.
	\end{lem}

	\begin{proof}
		From \Cref{lem:everywhere_growing} and the previous observations, it is enough to show that $\btau_T$ is everywhere growing.
		Let $(\alpha_{\mathrm{w_n}} \colon n \ge 0)$ and $(\alpha_{\mathrm{l_n}} \colon n \ge 0)$ be the sequences in $\cA$ consisting of the winners and losers of the iterates of the unlabelled Rauzy--Veech induction algorithm, respectively.
		From \Cref{thm:RV_adic} we have
		\[
			|\tau_0 \circ \tau_1 \circ \dotsb \circ \tau_n(\alpha_{\mathrm{l_n}})|
			= |\tau_0 \circ \tau_1 \circ \dotsb \circ \tau_{n-1}(\alpha_{\mathrm{l_n}})| + |\tau_0 \circ \tau_1 \circ \dotsb \circ \tau_{n-1}(\alpha_{\mathrm{w_n}})|,
		\]
		so it is enough to show that every letter is the loser for infinitely many steps of the unlabelled Rauzy--Veech induction algorithm.
		But this follows from \cite[Proposition 4.2]{BL09}.
	\end{proof}

	\section{Continuous eigenvalues of the natural symbolic codings} \label{sec:continuous_eigenvalues}

	To study the continuous eigenvalues of natural symbolic coding for linear involutions, we first describe their coboundaries.
	We then present a necessary condition for being a continuous eigenvalue of such a coding.
	Let $T$ be a linear involution on $I$ defined by the data $(\pi, \lambda)$ relative to the alphabet $\cA$ and let $\flat(\pi)$ be the labelled generalized permutation associated with $T$.
	For a finite set of row vectors $\{v_1, v_2, \dotsc, v_k\}$ we denote by $\langle v_1, v_2, \dotsc v_k \rangle$ its linear span.
	We denote by $\langle \lambda \rangle^\ort$ the subspace of row vectors $v$ in $\RR^\cA$ such that $\langle v, \lambda^\tr \rangle = 0$.
    
	\subsection{Coboundaries of the natural symbolic codings} For a word $w$ in the language of $\Omega_T$, a \emph{return word} to $w$ is a nonempty word $u$ such that $uw$ is in the language of $\Omega_T$ and $uw$ has exactly two occurrences of $w$, one as a prefix and one as a suffix.
	A \emph{letter-coboundary} on $\Omega_T$ is a morphism $c \colon F_\cA \to \RR$ (with respect to the additive structure of $\RR$) such that for every $\alpha \in \cA \cup \cA^{-1}$ and every return word $u$ to $\alpha$ in $\Omega_T$ we have $c(u) = 0$.
	We denote by $\cC(\Omega_T)$ the \emph{coboundary vector space} consisting of row vectors $(c(\alpha) \colon \alpha \in \cA \cup \cA^{-1})$, where $c$ is a letter-coboundary on $\Omega_T$.
	Consider $\cA_L$ and $\cA_R$ two copies of the alphabet $\cA \cup \cA^{-1}$ with canonical bijections $\alpha \mapsto \alpha_L$ and $\alpha \mapsto \alpha_R$ from $\cA \cup \cA^{-1}$ to $\cA_L$ and $\cA_R$, respectively.
	The \emph{extension graph} of the empty word $\Gamma_{\Omega_T}(\epsilon)$ is the bipartite undirected graph with vertex set $\cA_L \cup \cA_R$ and edge set defined as follows: if $\alpha_L \in \cA_L$ and $\beta_R \in \cA_R$, then $\alpha_L \beta_R$ is an edge if the word $\alpha \beta$ is in the language of $\Omega_T$.
	For a connected component $K$ of $\Gamma_{\Omega_T}(\epsilon)$ we define the row vectors $(e_K^L(\alpha) \colon \alpha \in \cA \cup \cA^{-1})$ and $(e_K^R(\alpha) \colon \alpha \in \cA \cup \cA^{-1})$ by
	\[
		e_K^L(\alpha)
		= \begin{cases}
			1 & \text{if $\alpha_L \in K$}\\
			0 & \text{otherwise}
		\end{cases}
		\quad \text{and} \quad
		e_K^R(\alpha)
		= \begin{cases}
			1 & \text{if $\alpha_R \in K$}\\
			0 & \text{otherwise.}
		\end{cases}
	\]
	From \cite[Proposition 6.1]{BCE26} we have that $\cC(\Omega_T)$ is spanned by the row vectors $e_K^L - e_K^R$, where $K$ ranges over the connected components of $\Gamma_{\Omega_T}(\epsilon)$.
	It turns out that $\cC(\Omega_T)$ only depends on the labelled generalized permutation $\flat(\pi)$ that defines the linear involution $T$:

	\begin{lem} \label{lem:c_pi}
		We have $\cC(\Omega_T) = \langle c_{\flat(\pi)} \rangle$, where $(c_{\flat(\pi)}(\alpha) \colon \alpha \in \cA \cup \cA^{-1})$ is defined by
		\[
			c_{\flat(\pi)}(\alpha)
			= \begin{cases}
				1  & \text{if $\alpha$ is duplicate in the top row of $\flat(\pi)$} \\
				-1 & \text{if $\alpha$ is duplicate in the bottom row of $\flat(\pi)$} \\
				0  & \text{otherwise.}
			\end{cases}
		\]
	\end{lem}

	\begin{proof}
	The graph $\Gamma_{\Omega_T}(\epsilon)$ decomposes in two trees $\cT_1$ and $\cT_2$ (this is a property shared by the so-called \emph{specular shifts}, which includes natural symbolic codings of linear involutions, see \cite[Theorem 8.4.9]{DP22}).
	We denote by $K_1$ (resp. $K_2$) the set of vertices of $\cT_1$ (resp. $\cT_2$); that is, $K_1$ and $K_2$ are the connected components of $\Gamma_{\Omega_T}(\epsilon)$.
	We partition $\cA_L$ in three sets $\cA_L^1, \cA_L^2, \cA_L^3$ consisting of letters $\alpha_L$ such that $\alpha$ is duplicate in the top row; letters $\alpha_L$ such that $\alpha$ is duplicate in the bottom row; and letters $\alpha_L$ such that $\alpha$ is simple, respectively.
	Similarly, we partition $\cA_R$ in three sets $\cA_R^1, \cA_R^2, \cA_R^3$ using analogous definitions.
	From \cite[Lemma 8.4.11]{DP22} we can suppose that if $\alpha_L$ belongs to $\cA_L^1$, then $\alpha_L, \alpha_L^{-1}\in K_1$; if $\alpha_L \in \cA_L^2$, then $\alpha_L, \alpha_L^{-1} \in K_2$; if $\alpha_L \in \cA_L^3$ and $\alpha$ is simple and occurs in the top row, then $\alpha_L \in K_2$ and $\alpha_L^{-1} \in K_1$; and if $\alpha$ is simple and occurs in the bottom row, then $\alpha_L \in K_1$ and $\alpha_L^{-1} \in K_2$.
    
    Next, from the same lemma, we can consider two cases:
	\begin{enumerate}
		\parindent=0pt
		\item If $\alpha_R$ belongs to $\cA_R^1$, then $\alpha_R, \alpha_R^{-1}$ belong to $K_1$; if $\alpha_R$ belongs to $\cA_R^2$, then $\alpha_R, \alpha_R^{-1}$ belong to $K_2$; if $\alpha_R$ belongs to $\cA_R^3$ and $\alpha$ is simple and occurs in the top row, then $\alpha_R$ belongs to $K_1$ and $\alpha_R^{-1}$ belongs to $K_2$; and if $\alpha$ is simple and occurs in the bottom row, then $\alpha_R$ belongs to $K_2$ and $\alpha_R^{-1}$ belongs to $K_1$.
		\item If $\alpha_R$ belongs to $\cA_R^1$, then $\alpha_R, \alpha_R^{-1}$ belong to $K_2$; if $\alpha_R$ belongs to $\cA_R^2$, then $\alpha_R, \alpha_R^{-1}$ belong to $K_1$; if $\alpha_R$ belongs to $\cA_R^3$ and $\alpha$ is simple and occurs in the top row, then $\alpha_R$ belongs to $K_2$ and $\alpha_R^{-1}$ belongs to $K_1$; and if $\alpha$ is simple and occurs in the bottom row, then $\alpha_R$ belongs to $K_1$ and $\alpha_R^{-1}$ belongs to $K_2$.
	\end{enumerate}
	In the first case, let $\alpha_L$ in $\cA_L^1$ (which exists since there exists duplicate letters in the top and bottom row).
	Since $\cT_1$ is connected, the only possibility is that there exists an edge $\alpha_L \beta_R$, for some $\beta_R$ in $\cA_R^3$.
	But then we have the edge $\beta_L^{-1} \alpha_R^{-1}$, which must be an edge in $\cT_2$ by \cite[Proposition 7.3.11]{DP22}.
	This contradicts the fact that $\beta_L^{-1}$ and $\alpha_R^{-1}$ belong to $K_1$.
	Hence we must be in the situation of the second case.
	A simple computation then shows that $c_{\flat(\pi)} = e_{K_1}^L - e_{K_1}^R = - (e_{K_2}^L - e_{K_2}^R)$.
	This finishes the proof since $\cC(\Omega_T)$ is spanned by the vectors $e_{K_1}^L - e_{K_1}^R$ and $e_{K_2}^L - e_{K_2}^R$.
    
	\end{proof}
	Observe that the row vector $c_{\flat(\pi)}$ is symmetric.
	By means of a standard computation, we have that if $\gamma = \pi \to \pi'$ is an edge in the unlabelled Rauzy diagram, then $c_{\flat(\pi)} M_{\flat(\gamma)} = c_{\flat(\pi')}$.
	Define $c_\pi = \pr(c_{\flat(\pi)})$ for every generalized permutation $\pi$, so that $c_\pi M_{\gamma} = c_{\pi'}$.

	\subsection{A necessary condition for being a continuous eigenvalue of the natural symbolic coding} Consider a linear involution $T$ defined by the data $(\pi, \lambda)$ that has the Keane property, so that $\RV^n(T)$ is well-defined for all $n \geq 0$.
	We say that the complex number $\exp(2 \pi i t)$ with $t \in [0,1)$ is a continuous eigenvalue of the minimal system $(\Omega_T, S)$ if there exists a continuous function $f \colon X \to \CC$, $f \neq 0$, such that $f \circ S = \exp(2 \pi i t) f$; $f$ is called a continuous eigenfunction of $(\Omega_T, S)$ associated with the eigenvalue $\exp(2 \pi i t)$.
	If $\exp(2 \pi i t)$ is an eigenvalue of the system with $t$ a rational number we say that $\exp(2 \pi i t)$ is a rational eigenvalue.
	The system $(\Omega_T, S)$ is topologically weakly-mixing if $(\Omega_T \times \Omega_T, S \times S)$ is \emph{transitive}, i.e., it has a dense orbit.
	Since the system $(\Omega_T, S)$ is minimal, it is well-known that it is topologically weakly-mixing if and only if it has no nonconstant continuous eigenfunctions \cite[Corollary 2.11]{KR69}.
	Recall that $\gamma(\pi, \lambda) = \gamma_0(\pi, \lambda) \gamma_1(\pi, \lambda) \dotsb $ is the infinite directed walk in the unlabelled Rauzy diagram, $\gamma^n(\pi, \lambda) = \gamma_0(\pi, \lambda) \gamma_1(\pi, \lambda) \dotsb \gamma_{n-1}(\pi, \lambda)$ for each $n \geq 1$, and that the stable space of $T$ is given by:
	\[
		\Es(\pi, \lambda)
		= \Big\{v \in \RR^\cA \colon \lim\limits_{n \to +\infty} v M_{\gamma^n(\pi, \lambda)} = 0\Big\}.
	\]
	Observe that $\Es(\pi, \lambda) \cap \langle c_\pi \rangle = \{0\}$ and that $\Es(\pi, \lambda) \oplus \langle c_\pi \rangle$ is a subspace of $\langle \lambda \rangle^\ort$.

	We now give a necessary condition for a complex number to be a continuous eigenvalue of $(\Omega_T, S)$:
	\begin{prop} \label{prop:necessary_condition}
		Let $T$ be a linear involution that has the Keane property defined by the data $(\pi, \lambda)$.
		If $\exp(2\pi i t)$ is a continuous eigenvalue of $(\Omega_T, S)$, then there exists a row integer vector $(m(\alpha) \colon \alpha \in \cA)$ such that $t \pr(e) + m \in \Es(\pi, \lambda) \oplus \langle c_\pi \rangle$. 
	\end{prop}

	\begin{proof}
		By \Cref{lem:primitive} and previous observations we have that $\btau_T$ is a primitive and recognizable directive sequence defined on a single alphabet $\cA \cup \cA^{-1}$.
		For integers $n_0$ and $n$ with $0 \leq n_0 < n$ define $M_{[n_0, n)} = M_{\flat(\gamma_{n_0}(\pi, \lambda))} M_{\flat(\gamma_{n_0 + 1}(\pi, \lambda))} \dotsb M_{\flat(\gamma_{n-1}(\pi, \lambda))}$.
		From \cite[Theorem 5.2]{BCE26} we have that if $\exp(2 \pi i t)$ is a continuous eigenvalue of $(\Omega_T, S)$ then there exist an integer $n_0 \geq 0$, a sequence of row vectors $(c_n \colon n \geq n_0)$ with $c_n \in \cC(\Omega_{\RV^n(T)})$ for each $n \geq n_0$; a sequence $(\nu_n \colon n \geq n_0)$ of row integer vectors; and a sequence of row vectors $(\eta_n \colon n \geq n_0)$ with $\lim\limits_{n \to +\infty} \eta_n = 0$ such that for all $n \geq n_0$ we have
		\[
			te M_{\flat(\gamma^n(\pi, \lambda))} + c_{n_0} M_{[n_0, n)}
			= \nu_n + \eta_n.
		\]
		Observe that $e M_{\flat(\gamma^n(\pi, \lambda))}$ and $c_{n_0} M_{[n_0, n)}$ are symmetric and that if $\alpha \in \cA \cup \cA^{-1}$ then $\lim\limits_{n \to +\infty} (\nu_n(\alpha) - \nu_n(\alpha^{-1})) = 0$, hence without loss of generality we can suppose that $\nu_n$ is symmetric.
		From this we see that
		\begin{align*}
			\nu_{n_0} M_{[n_0, n)} - \nu_n + \eta_{n_0} M_{[n_0, n)} - \eta_n
			&= te M_{\flat(\gamma^{n_0}(\pi, \lambda))} M_{[n_0, n)} + c_{n_0} M_{[n_0, n)}\\
			&\quad - (te M_{\flat(\gamma^n(\pi, \lambda))} + c_{n_0} M_{[n_0, n)})\\
			&= 0.
		\end{align*}
		Clearly, there exists a constant $\eta > 0$ such that if $\nu$ is an integer vector and $\| \nu \| < \eta$ then $\nu = 0$.
		Observe that if $\alpha \in \cA \cup \cA^{-1}$ then $\eta_{n_0}(\tau_{[n_0, n)}(\alpha)) = (\eta_{n_0} M_{[n_0, n)})(\alpha)$.
		Choose $0 \leq n_0 < n$ such that if $u$ is a word in the language of $X^{(n_0)}_{\btau_T}$ then $\| \eta_{n_0}(u) \| < \eta/2$ and $\| \eta_n \| < \eta / 2$.
		We deduce that $\nu_{n_0} M_{[n_0, n)} = \nu_n$ and $\eta_{n_0} M_{[n_0, n)} = \eta_n$ for all $n > n_0$.
		Define $c_0 = c_{n_0} M_{\flat(\gamma^{n_0}(\pi, \lambda))}^{-1}$, which belongs to $\langle c_{\flat(\pi)} \rangle$ by \Cref{lem:c_pi} and the remark made after, and let $m_0 = - \nu_{n_0} M_{\flat(\gamma^{n_0}(\pi, \lambda))}^{-1}$, which is a symmetric row integer vector since $\det(M_{\flat(\gamma^{n_0}(\pi, \lambda))}) = 1$.
		Then, we have that $\pr(c_0) \in \langle c_\pi \rangle$ and
		\begin{align*}
			\lim_{n \to +\infty} (te + c_0 + m_0) M_{\flat(\gamma^n(\pi, \lambda))}
			&= \lim_{n \to +\infty} (te M_{\flat(\gamma^n(\pi, \lambda))} + c_{n_0} M_{[n_0, n)} - \nu_{n_0} M_{[n_0, n)})\\
			&= \lim_{n \to +\infty} (te M_{\flat(\gamma^n(\pi, \lambda))} + c_{n_0} M_{[n_0, n)} - \nu_n)\\
			&= \lim_{n \to +\infty} \eta_n\\
			&= 0
		\end{align*}
		and thus $\lim\limits_{n \to +\infty} \pr(te + c_0 + m_0) M_{\gamma^n(\pi, \lambda)} = 0$, which proves the result with $m = \pr(m_0)$.
	\end{proof}

	From now on $e$ will denote the row vector whose entries are all equal to one in $\RR^\cA$, unless otherwise stated.
	Let $\Rec_\pi = \limsup\limits_{n \to +\infty} \{\lambda \in \cK_\pi \colon \RV^n(\pi, \lambda) = (\pi, \lambda^n)\}$.
	Clearly $\Rec_\pi$ is measurable in $\Delta_\pi$.
	From \cite[Theorem C]{BL09} we have that the (renormalized) Rauzy--Veech induction is recurrent.
	This implies that $\Leb_\pi(\Delta_\pi \setminus \Rec_\pi) = 0$. 

	\begin{prop} \label{prop:no_rational_eig}
		Let $\pi$ be a generalized permutation.
		\begin{itemize}
			\item If $\pi$ has at least one simple letter, then for almost all normalized admissible length vector $\lambda$ in $\Delta_\pi$ the natural symbolic coding of a linear involution defined by $(\pi, \lambda)$ does not posses a non-trivial rational eigenvalue.
			\item If all letters in $\pi$ are duplicate, then the natural symbolic coding of a linear involution defined by the data $(\pi, \lambda)$ has $-1$ as a rational eigenvalue.
		\end{itemize}
	\end{prop}

	\begin{proof}
		Let $\alpha$ be a simple letter in $\pi$.
		For $\lambda \in \Rec_\pi$ we know that there exists a strictly increasing sequence $(n_k \colon k \geq 0)$ such that $\RV^{n_k}(\pi, \lambda) = (\pi,\lambda^{n_k})$ for each $k \geq 0$.
		Fix $\lambda$ with this property and consider a linear involution $T$ defined by the data $(\pi, \lambda)$ such that $\exp(2 \pi i p/q)$ is a rational eigenvalue of $(\Omega_T, S)$ for some non-zero integers $p$ and $q$ with $\gcd(p, q) = 1$.
		By \Cref{prop:necessary_condition} there exist $c \in \langle c_\pi \rangle$ and a row integer vector $(m(\alpha) \colon \alpha \in \cA)$ such that $\lim\limits_{n \to +\infty} (pe + qc + qm)M_{\gamma^{n_k}(\pi, \lambda)} = 0$.
		Since $c M_{\gamma^{n_k}(\pi, \lambda)} = c$ we deduce that $qc$ is an integer vector and thus $(pe + qc + qm)M_{\gamma^{n_k}(\pi, \lambda)} = 0$ if $k$ is large enough, which implies that $pe + qc + qm = 0$.
		Therefore, $p = -qm(\alpha)$, which contradicts $\gcd(p, q) = 1$.

		If all letters in $\pi$ are duplicate, let $\flat(\pi)$ be a labelled version of $\pi$ and let $\cA_{\Rtop}$ (resp. $\cA_{\Rbot}$) be the letters in $\cA \cup \cA^{-1}$ that occur in the top (resp. bottom) row of $\flat(\pi)$.
		For a sequence $\omega = (x_n \colon n \in \ZZ)$ in $\Omega_T$ define the function $f \colon \Omega_T \to \{-1, 1\}$ as:
		\[
			f(\omega)
			= \begin{cases}
			+1 & \text{$x_0 \in \cA_{\Rtop}$} \\
			-1 & \text{$x_0 \in \cA_{\Rbot}$}.
			\end{cases}
		\]
		Then $f$ is a continuous eigenfunction of $(\Omega_T, S)$ associated with the eigenvalue $-1$.
	\end{proof}

	\section{Proof of \texorpdfstring{\Cref{thm:main}}{Theorem 1.1}} \label{sec:proof}

	We fix $s, r \ge 0$ integers and consider the generalized permutation $\sigma_{s,r}$, where
	\[
		\sigma_{s,r}
		=
		\begin{pmatrix}
			0 & A & 1 & 2 & \cdots & s & A & s+1 & s+2 & \cdots & s+r \\
			s+r & \cdots & s+2 & s+1 & B & s & \cdots & 2 & 1 & B & 0
		\end{pmatrix}.
	\]
	Note that the alphabet is $\cA = \{0, 1, \dotsc, s+r, A, B\}$.
	Let $\pi$ be a generalized permutation in the Rauzy class of $\sigma_{s,r}$ that has at least a simple letter.
	For a positive row integer vector $(h(\alpha) \colon \alpha \in \cA)$ and a row integer vector $(m(\alpha) \colon \alpha \in \cA)$ define the set
	\[
		\cS_{h, m}(\pi)
		= \{\lambda \in \Rec_{\pi} \colon t(\lambda) h + m \in \Es(\pi, \lambda) \oplus \langle c_\pi \rangle\},
	\]
	where $t(\lambda)$ is the unique value such that $t(\lambda) h + m \in \langle \lambda \rangle^\ort$, that is, $t(\lambda) = \langle -m, \lambda^\tr \rangle / \langle h, \lambda^\tr \rangle$.
	We say that the pair $(h, m)$ is \emph{trivial} for $\pi$ if there exists $r \in \QQ$ such that $m = rh$ or if $m \in \langle c_\pi \rangle$.
	In the first case, the condition $t(\lambda)h + m \in \langle \lambda \rangle^\ort$ says $t(\lambda)h + m = 0$ and consequently $t(\lambda)$ is rational; and in the second case $t(\lambda) = 0$ since $c_\pi \in \langle \lambda \rangle^\ort$.
	Observe that if $\bar{\gamma}$ is a directed walk that joins $\pi$ and $\pi'$ in the unlabelled Rauzy diagram, then $(h, m)$ is trivial for $\pi$ if and only if $(h M_{\bar{\gamma}}, m M_{\bar{\gamma}})$ is trivial for $\pi'$,

	From \Cref{prop:necessary_condition} and \Cref{prop:no_rational_eig} we have that the set of all admissible length vectors $\lambda \in \Rec_{\pi}$ such that the natural symbolic coding $(\Omega_T, S)$ of a linear involution $T$ defined by the data $(\pi, \lambda)$ is not topologically weakly mixing is contained in the union of the sets $\cS_{e, m}(\pi)$, where $m$ ranges over all row integer vectors such that $(e, m)$ is non-trivial for $\pi$.
	Therefore, to prove \Cref{thm:main} it is enough to prove that $\Leb_{\pi}(\cS_{e, m}(\pi)) = 0$ for all such pairs $(e, m)$.

	\begin{lem} \label{lem:measurable}
		If $(h, m)$ is a non-trivial pair for $\pi$ then the set $\cS_{h,m}(\pi)$ is $\Leb_{\pi}$--measurable.
	\end{lem}

	\begin{proof}
		Clearly the map $\lambda \mapsto t(\lambda)$ is measurable.
		If $\lambda \in \cS_{h,m}(\pi)$ there exists a strictly increasing sequence $(n_k \colon k \geq 0)$ such that $\RV^{n_k}(\pi, \lambda) = (\pi,\lambda^{n_k})$ and $c(\lambda) \in \langle c_\pi \rangle$ such that $\lim\limits_{k \to +\infty} (t(\lambda)h + m - c(\lambda))M_{\gamma^{n_k}(\pi, \lambda)} = 0$.
		Since $c(\lambda) M_{\gamma^{n_k}(\pi, \lambda)} = c(\lambda)$ we deduce that $c(\lambda) = \lim\limits_{k \to +\infty} (t(\lambda)h + m)M_{\gamma^{n_k}(\pi, \lambda)}$ and hence the map $\lambda \mapsto c(\lambda)$ is measurable.
		Thus $\lambda$ is in $\cS_{h,m}(\pi)$ if and only if $\lim\limits_{n \to +\infty} (t(\lambda)h + m - c(\lambda))M_{\gamma^n(\pi, \lambda)} = 0$.
		From this we see that the set of such $\lambda$'s is measurable.
	\end{proof}

	\subsection{Distortion estimates} We follow \cite[Section 6.16]{BDG+24}.
	For a directed walk $\bar{\gamma}$ that joins $\pi$ and $\pi'$ in the unlabelled Rauzy diagram define the bijection $\R_{\bar{\gamma}} \colon \Delta_{\pi'} \to \Delta_\pi$ by
	\[
		\R_{\bar{\gamma}}(x')
		= \frac{M_{\bar{\gamma}} x'}{\| M_{\bar{\gamma}} x' \|_1}.
	\]
	Let $\cJ(\R_{\bar{\gamma}})$ be the Jacobian of $\R_{\bar{\gamma}}$, that is,
	\[
		\cJ(\R_{\bar{\gamma}})(x')
		= \frac{\diff \Leb_{\pi}}{\diff [(\R_{\bar{\gamma}})_\ast \Leb_{\pi'}]}(\R_{\bar{\gamma}}(x')).
	\]
	For $K \geq 1$ we say that $\bar{\gamma}$ has \emph{$K$--bounded distortion} if
	\[
		\frac{\cJ(\R_{\bar{\gamma}})(x)}{\cJ(\R_{\bar{\gamma}})(x')}
		\leq K	
	\]
	for any $x$ and $x'$ in $\Delta_{\pi'}$.

	\begin{lem} \label{lem:bd_distortion}
		Suppose that $\bar{\gamma}$ is a directed walk that joins $\pi$ and $\pi'$ in the unlabelled Rauzy diagram which has $K$--bounded distortion.
		If $U$ is any measurable set in $\Delta_{\pi'}$, then we have
		\[
			\frac{1}{K^2} \frac{\Leb_{\pi'}(U)}{\Leb_{\pi'}(\Delta_{\pi'})}
			\leq \frac{\Leb_\pi(\R_{\bar{\gamma}}(U))}{\Leb_\pi(\R_{\bar{\gamma}}(\Delta_{\pi'}))}
			\leq K^2 \frac{\Leb_{\pi'}(U)}{\Leb_{\pi'}(\Delta_{\pi'})}.
		\]
	\end{lem}

	\begin{proof}
		Let $\mu = \Leb_\pi$ and $\nu = (\R_{\bar{\gamma}})_\ast \Leb_{\pi'}$.
		Fix $x_0$ a point in $\Delta_{\pi'}$ and let $V$ be a measurable set in $\Delta_{\pi'}$.
		From the change of coordinates formula, we have
		\begin{align*}
			\Leb_\pi(\R_{\bar{\gamma}}(V))
			&= \int_{\R_{\bar{\gamma}}(V)} \diff \mu(x) \\
			&= \int_{\R_{\bar{\gamma}}(V)} \frac{\diff \mu}{\diff \nu}(x) \diff \nu(x) \\
			&= \int_V \frac{\diff \mu}{\diff \nu}(\R_{\bar{\gamma}}(x')) \diff \nu(\R_{\bar{\gamma}}(x')) \\
			&= \int_V \cJ(\R_{\bar{\gamma}})(x') \diff \nu (\R_{\bar{\gamma}}(x')) \\
			&= \int_V \cJ(\R_{\bar{\gamma}})(x') \diff \Leb_{\pi'} (x').
		\end{align*}
		From this and the bounds on the Jacobian we obtain
		\[
			\frac{1}{K} \cJ(\R_{\bar{\gamma}})(x_0)
			\leq \frac{\Leb_\pi(\R_{\bar{\gamma}}(V))}{\Leb_{\pi'}(V)}
			\leq K \cJ(\R_{\bar{\gamma}})(x_0).
		\]
		If we use the above estimate with $V = U$ and then with $V = \Delta_{\pi'}$ we obtain
		\[
			\frac{1}{K} \cJ(\R_{\bar{\gamma}})(x_0)
			\leq \frac{\Leb_\pi(\R_{\bar{\gamma}}(U))}{\Leb_{\pi'}(U)}
			\leq K \cJ(\R_{\bar{\gamma}})(x_0),
		\]
		and
		\[
			\frac{1}{K} \cJ(\R_{\bar{\gamma}})(x_0)
			\leq \frac{\Leb_\pi(\R_{\bar{\gamma}}(\Delta_{\pi'}))}{\Leb_{\pi'}(\Delta_{\pi'})}
			\leq K \cJ(\R_{\bar{\gamma}})(x_0).
		\]
		The conclusion follows.

	\end{proof}

	The Jacobian indeed can be computed: from \cite[Equation 8.3 and Lemma 8.1]{Gad12} we know that there exists a constant $a_{\bar{\gamma}} > 0$ such that
	\[
		\cJ(\R_{\bar{\gamma}})(x')
		= \frac{a_{\bar{\gamma}}}{\| M_{\bar{\gamma}} x' \|_1^{|\cA| - 1}}.
	\]
	For a letter $\alpha \in \cA$ we define $\col_\alpha(M_{\bar{\gamma}})$ the $\alpha$--column of the matrix $M_{\bar{\gamma}}$.
	Suppose that $k \geq 1$ is a constant.
	The matrix $M_{\bar{\gamma}}$ is said to be \emph{$k$--balanced} if
	\[
		\frac{\| \col_{\alpha}(M_{\bar{\gamma}}) \|_1}{\| \col_{\alpha'}(M_{\bar{\gamma}}) \|_1}
		\leq k
	\]
	for any $\alpha$ and $\alpha'$ in $\cA$.

	\begin{lem} \label{lem:balanced}
		Suppose that $M_{\bar{\gamma}}$ is $k$--balanced.
		Then $\bar{\gamma}$ has $k^{|\cA| - 1}$--bounded distortion.
	\end{lem}

	\begin{proof}
		For any $x'$ in $\Delta_{\pi'}$ we have
		\[
			\| M_{\bar{\gamma}} x' \|_1
			= \sum_{\alpha \in \cA} x'_\alpha \| \col_{\alpha} (M_{\bar{\gamma}})\|_1.
		\]
		Since $\sum_{\alpha \in \cA} x'_\alpha = 1$ it follows that
		\[
			\min_{\alpha \in \cA} \| \col_{\alpha}(M_{\bar{\gamma}}) \|_1
			\leq \| M_{\bar{\gamma}} x' \|_1
			\leq \max_{\alpha \in \cA} \| \col_{\alpha}(M_{\bar{\gamma}}) \|_1.
		\]
		Therefore, for any $x$ and $x'$ in $\Delta_{\pi'}$ we obtain
		\[
			\frac{\cJ(\R_{\bar{\gamma}})(x)}{\cJ(\R_{\bar{\gamma}})(x')}
			= \frac{\| M_{\bar{\gamma}} x' \|_1^{|\cA| - 1}}{\| M_{\bar{\gamma}} x \|_1^{|\cA| - 1}}
			\leq \frac{\max_{\alpha \in \cA} \| \col_{\alpha}(M_{\bar{\gamma}}) \|_1^{|\cA| - 1}}{\min_{\alpha \in \cA} \| \col_{\alpha}(M_{\bar{\gamma}}) \|_1^{|\cA| - 1}}
			\leq k^{|\cA| - 1}.
		\]
	\end{proof}

	\begin{lem} \label{lem:positive_matrix}
		Suppose that $\bar{\gamma} = \bar{\gamma}_1 \bar{\gamma}_2$ is a concatenation of two directed walks $\bar{\gamma}_1$ and $\bar{\gamma}_2$ in the unlabelled Rauzy diagram.
		Suppose that $M_{\bar{\gamma}_2}$ is a positive matrix.
		Then there exists a constant $k(\bar{\gamma}_2) \geq 1$ such that $M_{\bar{\gamma}}$ is $k(\bar{\gamma}_2)$--balanced.
	\end{lem}

	\begin{proof}
		Let $k(\bar{\gamma}_2) = \max_{\alpha, \beta \in \cA} M_{\bar{\gamma}_2}(\alpha, \beta) / \min_{\alpha, \beta \in \cA} M_{\bar{\gamma}_2}(\alpha, \beta)$.
		For $\alpha$ and $\alpha'$ in $\cA$ we have
		\begin{align*}
			\| \col_\alpha(M_{\bar{\gamma}}) \|_1
			&= \sum_{\beta \in \cA} M_{\bar{\gamma_1} \bar{\gamma_2}}(\beta, \alpha) \\
			&= \sum_{\beta \in \cA} \sum_{\gamma \in \cA} M_{\bar{\gamma}_1}(\beta, \gamma) M_{\bar{\gamma}_2}(\gamma, \alpha) \\
			&\leq k(\bar{\gamma}_2) \sum_{\beta \in \cA} \sum_{\gamma \in \cA} M_{\bar{\gamma}_1}(\beta, \gamma) M_{\bar{\gamma}_2}(\gamma, \alpha') \\
			&= k(\bar{\gamma}_2) \| \col_{\alpha'}(M_{\bar{\gamma}}) \|_1.
		\end{align*}
	\end{proof}

	\begin{lem} \label{lem:image}
		Suppose that $(h, m)$ is a non-trivial pair for $\pi$ and let $\bar{\gamma}$ be a directed walk that joins $\pi$ and $\pi'$ in the unlabelled Rauzy diagram.
		Then
		\[
			\R_{\bar{\gamma}}(\cS_{h M_{\bar{\gamma}}, m M_{\bar{\gamma}}}(\pi'))
			= \cS_{h, m}(\pi).
		\]
	\end{lem}

	\begin{proof}
		Let $\lambda \in \R_{\bar{\gamma}}(\cS_{h M_{\bar{\gamma}}, m M_{\bar{\gamma}}}(\pi'))$.
		Then there exists $\lambda' \in \cS_{h M_{\bar{\gamma}}, m M_{\bar{\gamma}}}(\pi')$ such that $\lambda = \R_{\bar{\gamma}}(\lambda')$.
		From this we see that $\Es(\pi', \lambda') = \Es(\pi, \lambda) M_{\bar{\gamma}}$.
		Then $t(\lambda') h M_{\bar{\gamma}} + m M_{\bar{\gamma}} \in \Es(\pi', \lambda') \oplus \langle c_{\pi'} \rangle$ implies $t(\lambda') h + m \in \Es(\pi, \lambda) \oplus \langle c_\pi \rangle$.
		That is, $\lambda \in \cS_{h, m}(\pi)$.
		Conversely, if $\lambda \in \cS_{h, m}(\pi)$ then $t(\lambda) h + m \in \Es(\pi, \lambda) \oplus \langle c_\pi \rangle$ with $t(\lambda) = - \langle m, \lambda^\tr \rangle / \langle h, \lambda^\tr \rangle$.
		Let $\lambda' = \R_{\bar{\gamma}}^{-1}(\lambda)$.
		We obtain $\R_{\bar{\gamma}}(\lambda') = \lambda$ and $\Es(\pi', \lambda') = \Es(\pi, \lambda) M_{\bar{\gamma}}$.
		Then $t(\lambda) h M_{\bar{\gamma}} + m M_{\bar{\gamma}} \in \Es(\pi', \lambda') \oplus \langle c_{\pi'} \rangle$.
		That is, $\lambda' \in \cS_{h M_{\bar{\gamma}}, m M_{\bar{\gamma}}}(\pi')$ and $\lambda \in \R_{\bar{\gamma}}(\cS_{h M_{\bar{\gamma}}, m M_{\bar{\gamma}}}(\pi))$.
	\end{proof}

	From \Cref{lem:bd_distortion} and \Cref{lem:image} we immediately obtain the following:

	\begin{cor} \label{cor:sufficient}
		If $\bar{\gamma}$ is a directed walk that joins $\pi$ and $\sigma_{s,r}$ in the unlabelled Rauzy diagram, $(e,m)$ a non-trivial pair for $\pi$ and $\Leb_{\sigma_{s,r}}(\cS_{e M_{\bar{\gamma}}, m M_{\bar{\gamma}}}(\sigma_{s,r})) = 0$, then $\Leb_\pi(\cS_{e, m}(\pi)) = 0$.
	\end{cor}

	Let $e' = e M_{\bar{\gamma}}$ and $m' = m M_{\bar{\gamma}}$.
	From now on we suppose that $\Leb_{\sigma_{s,r}}(\cS_{e', m'}(\sigma_{s,r})) > 0$ and we aim to find a contradiction if the genus is sufficiently large.

	\begin{lem} \label{lem:density_point}
		If $\Leb_{\sigma_{s,r}}(\cS_{e', m'}(\sigma_{s,r})) > 0$, then for all $\epsilon > 0$ there exists a cycle $\bar{\gamma}'$ at $\sigma_{s,r}$ in the unlabelled Rauzy diagram such that $(\bar{h}, \bar{m}) = (e' M_{\bar{\gamma}'}, m' M_{\bar{\gamma}'})$ is a non-trivial pair for $\sigma_{s,r}$ and $\Leb_{\sigma_{s,r}}(\Delta_{\sigma_{s,r}} \setminus \cS_{\bar{h}, \bar{m}}(\sigma_{s,r})) < \epsilon$.
	\end{lem}

	\begin{proof}
		Let $\epsilon > 0$.
		From \Cref{lem:primitive} (see also \cite[Lemma 3.4]{AR12} in the context of interval exchange transformations with involutions) we know that there exists a cycle $\bar{\gamma}_\ast$ at $\sigma_{s, r}$ in the unlabelled Rauzy diagram such that $M_{\bar{\gamma}_\ast}$ is a positive matrix.
		Since the Rauzy--Veech induction is recurrent, we know that for $\Leb_{\sigma_{s,r}}$--almost all $\lambda$ in $\cS_{e', m'}(\sigma_{s,r})$ there exists a strictly increasing sequence $(n_k \colon k \geq 0)$ such that
		\[
			\gamma_{n_k - |\bar{\gamma}_\ast|}(\sigma_{s,r}, \lambda) \dotsb \gamma_{n_k - 2}(\sigma_{s,r}, \lambda) \gamma_{n_k - 1}(\sigma_{s,r}, \lambda)
			= \bar{\gamma}_\ast,
		\]
		where $|\bar{\gamma}_\ast|$ is the length of $\bar{\gamma}_\ast$ and $\gamma(\sigma_{s, r}, \lambda) = \gamma_0(\sigma_{s, r}, \lambda) \gamma_1(\sigma_{s, r}, \lambda) \dotsc$ is the infinite directed walk in the unlabelled Rauzy diagram.
		For each $\lambda$ as before and $k \geq 0$ define $\Delta_{\sigma_{s,r}}^{n_k}(\lambda)$ as $\R_{\gamma^{n_k}(\sigma_{s,r}, \lambda)}(\Delta_{\sigma_{s,r}})$, which contains $\lambda$.
		It is a classical fact that since the matrix $M_{\bar{\gamma}_\ast}$ is positive the diameter of $\Delta_{\sigma_{s,r}}^{n_k}(\lambda)$ converge to $0$ as $k$ tends to $+\infty$, see for example \cite[Section 15.2]{Fur60} (this is exactly the criterion for unique ergodicity of interval exchange transformations \cites{Mas82, Vee82}, see also \cite{DN88} for unique ergodicity of linear involutions).
		Therefore, for $\Leb_{\sigma_{s,r}}$--almost all $\lambda$ in $\cS_{e', m'}(\sigma_{s,r})$ we have
		\[
		\lim_{k \to +\infty} \frac{\Leb_{\sigma_{s,r}}(\Delta_{\sigma_{s,r}}^{n_k}(\lambda) \cap \cS_{e', m'}(\sigma_{s,r}))}{\Leb_{\sigma_{s,r}}(\Delta_{\sigma_{s,r}}^{n_k}(\lambda))}
		= 1.
		\]
		From \Cref{lem:balanced} and \Cref{lem:positive_matrix} we know that there exists a constant $K_\ast \geq 1$ such that for every $k \geq 0$ the cycle $\gamma^{n_k}(\sigma_{s,r}, \lambda)$ has $K_\ast$--bounded distortion.
		Fix $\lambda \in \cS_{e' , m'}(\sigma_{s,r})$ and $k \geq 0$ such that
		\[
			\frac{\Leb_{\sigma_{s,r}}(\Delta_{\sigma_{s,r}}^{n_k}(\lambda) \setminus \cS_{e', m'}(\sigma_{s,r}))}{\Leb_{\sigma_{s,r}}(\Delta_{\sigma_{s,r}}^{n_k}(\lambda))}
			< \frac{\epsilon}{K_\ast^2 \Leb_{\sigma_{s,r}}(\Delta_{\sigma_{s,r}})}
		\]
		and let $\bar{\gamma}' = \gamma^{n_k}(\sigma_{s,r}, \lambda)$.
		Since $\bar{\gamma}'$ is a cycle at $\sigma_{s,r}$ we know that $(\bar{h}, \bar{m}) = (e' M_{\bar{\gamma}'}, m' M_{\bar{\gamma}'})$ is a non-trivial pair for $\sigma_{s,r}$.
		From \Cref{lem:image} we have $\R_{\bar{\gamma}'} (\cS_{e' M_{\bar{\gamma}'}, m' M_{\bar{\gamma}'}}(\sigma_{s,r})) = \cS_{e', m'}(\sigma_{s,r})$.
		Since $\R_{\bar{\gamma}'}$ is a bijection on $\Delta_{\sigma_{s,r}}$, from \Cref{lem:bd_distortion} we obtain
		\begin{align*}
			\frac{\Leb_{\sigma_{s,r}}(\Delta_{\sigma_{s,r}} \setminus \cS_{e' M_{\bar{\gamma}'}, m' M_{\bar{\gamma}'}}(\sigma_{s,r}))}{\Leb_{\sigma_{s,r}}(\Delta_{\sigma_{s,r}})}
			&\leq K_\ast^2 \frac{\Leb_{\sigma_{s,r}}(\R_{\bar{\gamma}'}(\Delta_{\sigma_{s,r}} \setminus \cS_{e' M_{\bar{\gamma}'}, m' M_{\bar{\gamma}'}}(\sigma_{s,r})))}{\Leb_{\sigma_{s,r}}(\R_{\bar{\gamma}'}(\Delta_{\sigma_{s,r}}))} \\
			&\leq K_\ast^2 \frac{\Leb_{\sigma_{s,r}}(\R_{\bar{\gamma}'}(\Delta_{\sigma_{s,r}}) \setminus \cS_{e', m'}(\sigma_{s,r}))}{\Leb_{\sigma_{s,r}}(\R_{\bar{\gamma}'}(\Delta_{\sigma_{s,r}}))} \\
			&= K_\ast^2 \frac{\Leb_{\sigma_{s,r}}(\Delta_{\sigma_{s,r}}^{n_k}(\lambda) \setminus \cS_{e', m'}(\sigma_{s,r}))}{\Leb_{\sigma_{s,r}}(\Delta_{\sigma_{s,r}}^{n_k}(\lambda))} \\
			&< \frac{\epsilon}{\Leb_{\sigma_{s,r}}(\Delta_{\sigma_{s,r}})}
		\end{align*}
		and thus $\Leb_{\sigma_{s,r}}(\Delta_{\sigma_{s,r}} \setminus \cS_{e' M_{\bar{\gamma}'}, m' M_{\bar{\gamma}'}}) < \epsilon$, as we wanted.
	\end{proof}

	\begin{prop} \label{prop:intersection}
		Suppose that $\Leb_{\sigma_{s,r}}(\cS_{e', m'}(\sigma_{s,r})) > 0$.
		If $\cC$ is a finite set of cycles at $\sigma_{s,r}$ in the unlabelled Rauzy diagram, then there exists a non-trivial pair $(\bar{h}, \bar{m}) = (e' M_{\bar{\gamma}'}, m' M_{\bar{\gamma}'})$ for $\sigma_{s,r}$, for some cycle $\bar{\gamma}'$ at $\sigma_{s,r}$ such that
		\[
			\Leb_{\sigma_{s,r}} \Bigg(\cS_{\bar{h}, \bar{m}}(\sigma_{s,r}) \cap \bigcap_{\delta \in \cC} \cS_{\bar{h}M_{\delta}, \bar{m}M_{\delta}}(\sigma_{s,r})\Bigg)
			> 0.
		\]
	\end{prop}

	\begin{proof}
		Since $\cC$ is a finite set of cycles and $\Delta_{\sigma_{s,r}}$ is compact, there exists a constant $K_\cC \geq 1$ such that each $\delta$ in $\cC$ has $K_\cC$--bounded distortion.
		Let $\eta = \Leb_{\sigma_{s,r}}(\Delta_{\sigma_{s,r}})$ and choose
		\[
			\epsilon < \min \Bigg\{ \frac{\eta}{2}, \frac{\min_{\delta \in \cC} \Leb_{\sigma_{s,r}}(\R_\delta(\Delta_{\sigma_{s,r}}))}{2 |\cC| K_\cC^2}\Bigg\}.
		\]
		From \Cref{lem:density_point} there exists a cycle $\bar{\gamma}'$ at $\sigma_{s,r}$ such that $(\bar{h}, \bar{m}) = (e' M_{\bar{\gamma}}, m' M_{\bar{\gamma}})$ is a non-trivial pair for $\sigma_{s,r}$ and $\Leb_{\sigma_{s,r}}(\Delta_{\sigma_{s,r}} \setminus \cS_{\bar{h}, \bar{m}}(\sigma_{s,r})) < \epsilon$.
		From \Cref{lem:bd_distortion} and \Cref{lem:image} we obtain the estimate
		\begin{align*}
			\Leb_{\sigma_{s,r}} \Bigg( \bigcup_{\delta \in \cC} \Delta_{\sigma_{s,r}} \setminus \cS_{\bar{h} M_{\delta}, \bar{m} M_{\delta}}(\sigma_{s,r}) \Bigg )
			&\leq \sum_{\delta \in \cC} \frac{\Leb_{\sigma_{s,r}}(\Delta_{\sigma_{s,r}} \setminus \cS_{\bar{h} M_{\delta}, \bar{m} M_{\delta}}(\sigma_{s,r}))}{\Leb_{\sigma_{s,r}}(\Delta_{\sigma_{s,r}})} \eta \\
			&\leq K_\cC^2 \eta \sum_{\delta \in \cC} \frac{\Leb_{\sigma_{s,r}}(\Delta_{\sigma_{s,r}} \setminus \cS_{\bar{h}, \bar{m}}(\sigma_{s,r}))}{\Leb_{\sigma_{s,r}} (\R_{\delta}(\Delta_{\sigma_{s,r}}))} \\
			&\leq K_\cC^2 \eta |\cC| \frac{\Leb_{\sigma_{s,r}}(\Delta_{\sigma_{s,r}} \setminus \cS_{\bar{h}, \bar{m}}(\sigma_{s,r}))}{\min_{\delta \in \cC} \Leb_{\sigma_{s,r}}(\R_\delta(\Delta_{\sigma_{s,r}}))} \\
			&< K_\cC^2 \eta |\cC| \frac{\epsilon}{\min_{\delta \in \cC} \Leb_{\sigma_{s,r}}(\R_\delta(\Delta_{\sigma_{s,r}}))} \\
			&< \frac{\eta}{2}.
		\end{align*}
		Therefore, since $\Leb_{\sigma_{s,r}} (\cS_{\bar{h}, \bar{m}}(\sigma_{s,r})) > \eta - \epsilon$ we deduce
		\[
			\Leb_{\sigma_{s,r}} \Bigg(\cS_{\bar{h}, \bar{m}}(\sigma_{s,r}) \cap \bigcap_{\delta \in \cC} \cS_{\bar{h}M_{\delta}, \bar{m}M_{\delta}}(\sigma_{s,r})\Bigg)
			> \frac{\eta}{2} - \epsilon
			> 0.
		\]
	\end{proof}

	\subsection{End of the proof of \texorpdfstring{\Cref{thm:main}}{Theorem 1.1}} For $m \geq 0$ we denote by $\Rtop^m$ (resp. $\Rbot^m$) the concatenation of $m$ top (resp. bottom) operations.
	The following computations were performed using the \texttt{surface\_dynamics} package for SageMath \cite{SageMath}.
	We refer to \cite[Section 5.2]{Arb18} for a proof.

	\begin{lem} \label{lem:linear_action}
		For each $k$ in $\{1, 2, \dotsc, s+r\} \setminus \{s+1\}$ define the directed walk $\delta_k$ that starts at $\sigma_{s,r}$ and where its sequence of top and bottom operations is
		\[
			\Rtop^{k+1} \Rbot^{s+r+1 - k} \Rtop^{s+r+1-k}.
		\]
		Define $\delta_{\Rtop}$ (resp. $\delta_{\Rbot}$) to be the directed walk that starts at $\sigma_{s,r}$, where its sequence of top and bottom operations is $\Rtop^{s+r+2}$ (resp. $\Rbot^{s+r+2}$).
		If $s \neq 0$ define $\delta_{s,r}$ to be the directed walk that starts at $\sigma_{s,r}$, where its sequence of top and bottom operations is
		\[
			\Rbot^{s+r} \Rtop \Rbot \Rtop \Rbot^{r+2} \Rtop \Rbot^s \Rtop^s \Rbot.
		\]
		If $s = 0$ and $r \neq 0$ define $\delta_{0,r}$ to be the directed walk that starts at $\sigma_{0,r}$, where its sequence of top and bottom operations is
		\[
			\Rtop \Rbot^{r} \Rtop^2 \Rbot^2 \Rtop^{r+1} \Rbot^2 \Rtop^2 \Rbot \Rtop^{r+1}.
		\]
		Then all the previously defined directed walks are cycles at $\sigma_{s,r}$ and their linear action on row vectors are given as follows:
		\begin{align*}
			z M_{\delta_k}
			&= z - z(s+r) v_{s+r} + z(k)(v_k - v_{s+r})
			&&\quad 1 \leq k \leq s, \\
			z M_{\delta_k}
			&= z - z(s+r) v_{s+r} + z(k-1)(v_{k-1} - v_{s+r})
			&&\quad s+2 \leq k \leq s+r, \\
			z M_{\delta_{\Rtop}}
			&= z - z(s+r) v_{s+r}, \\
			z M_{\delta_{\Rbot}}
			&= z + z(0) v_{0}, \\
			z M_{\delta_{s,r}}
			&= z + (3z(0) + 2z(1) + z(A) + z(B))v_0 \\
			&\quad + (z(0) + z(1))v_1 \\
			&\quad - (z(0) + z(1) + (z(A) + z(B)) / 2) (v_A + v_B), \\
			z M_{\delta_{0,r}}
			&= z - (4z(1) + 6z(r+1) + 6z(r+2))v_1 \\
			&\quad - (z(r) + 2z(r+1) + 2z(r+2))v_r \\
			&\quad + (2z(1) + 4z(r+1) + 4z(r+2))(v_A + v_B),
		\end{align*}
		where $v_0, v_1, \dotsc, v_{s+r}, v_A, v_B$ are the rows of $\Omega_{\sigma_{s,r}}$.
	\end{lem}

	\begin{lem} \label{lem:non_trivial}
		Let $\bar{\gamma}'$ be a cycle at $\sigma_{s,r}$ in the unlabelled Rauzy diagram, $r \in \QQ$, $\bar{h} = e' M_{\bar{\gamma}'}$ and $\bar{m} = m' M_{\bar{\gamma}'}$.
		If $(r \bar{h} + \bar{m}) \in \Es(\sigma_{s,r}, \lambda) \oplus \langle c_{\sigma_{s,r}} \rangle$ for some $\lambda \in \Rec_{\sigma_{s,r}}$, then the pair $(\bar{h}, \bar{m})$ is trivial for $\sigma_{s,r}$.
	\end{lem}

	\begin{proof}
		Assume first that $r \neq 0$ and write $r = p/q$ for some non-zero integers $p, q$ with $\gcd(p,q) = 1$.
		There exists a strictly increasing sequence $(n_k \colon k \geq 0)$ such that $\RV^{n_k}(\sigma_{s,r}, \lambda) = (\sigma_{s,r}, \lambda^{n_k})$ for each $k \geq 0$.
		From the condition $(r \bar{h} + \bar{m}) \in \Es(\sigma_{s,r}, \lambda) \oplus \langle c_{\sigma_{s,r}} \rangle$ we see that there exists $c \in \langle c_{\sigma_{s,r}} \rangle$ such that $r \bar{h} + \bar{m} - c \in \Es(\sigma_{s,r}, \lambda)$.
		Since $c = c M_{\gamma^{n_k}(\sigma_{s,r}, \lambda)}$ for all $k \ge 0$ we have $qc = \lim\limits_{k \to +\infty} (p \bar{h} + q \bar{m}) M_{\gamma^{n_k}(\sigma_{s,r}, \lambda)}$, thus $qc$ is an integer vector.
		From $\lim\limits_{k \to +\infty} (p \bar{h} + q \bar{m} - qc)M_{\gamma^{n_k}(\sigma_{s,r}, \lambda)} = 0$ we deduce that $(p \bar{h} + q \bar{m} - qc)M_{\gamma^{n_k}(\sigma_{s,r}, \lambda)} = 0$ if $k$ is sufficiently large, which implies that $p \bar{h} + q \bar{m} - qc = p e' M_{\bar{\gamma}'} + q m' M_{\bar{\gamma}'} - q c M_{\bar{\gamma}'} = 0$ and so $p e' + q m' - qc = 0$.
		Since $e' = e M_{\bar{\gamma}}$, $m' = m M_{\bar{\gamma}}$ and $c = c' M_{\bar{\gamma}}$ for some $c' \in \langle c_\pi \rangle$ we obtain $pe + qm - qc' = 0$.
		If $\alpha \in \cA$ is a simple letter in $\pi$, we obtain $p = - qm(\alpha)$, which contradicts $\gcd(p,q) = 1$.
		This shows that $r = 0$, but then $m = c'$, $m' = c$ and $\bar{m} = c$.
		Hence the pair $(\bar{h}, \bar{m})$ is trivial for $\sigma_{s,r}$.
	\end{proof}

	Let $\cC$ be the set of cycles $\delta_{\Rtop}$, $\delta_{\Rbot}$, $\delta_k$ for $k$ in $\{1, 2, \dotsc, s+r\} \setminus \{s+1\}$ and $\delta_{s,r}$ if $s \neq 0$ or $\delta_{0,r}$ if $s = 0$ and $r \neq 0$.
	From \Cref{prop:intersection}, if $\Leb_{\sigma_{s,r}}(\cS_{e', m'}(\sigma_{s,r})) > 0$ there exists a non-trivial pair $(\bar{h}, \bar{m}) = (e' M_{\bar{\gamma}'}, m' M_{\bar{\gamma}'})$ for $\sigma_{s,r}$ for some cycle $\bar{\gamma}'$ at $\sigma_{s,r}$ such that $\Leb_{\sigma_{s,r}}(\cS) > 0$, where
	\[
		\cS
		= \cS_{\bar{h}, \bar{m}}(\sigma_{s,r}) \cap \bigcap_{\delta \in \cC} \cS_{\bar{h}M_{\delta}, \bar{m}M_{\delta}}(\sigma_{s,r}).
	\]

	\begin{prop} \label{prop:inclusion}
		Let $v_0, v_1, \dotsc, v_{s+r}, v_A, v_B$ be the rows of the matrix $\Omega_{\sigma_{s,r}}$.
		If $\lambda \in \cS$, then
		\[
			\langle v_0, v_1, \dotsc, v_{s+r}, v_A + v_B \rangle \cap \langle \lambda \rangle^\ort
			\subseteq \Es(\sigma_{s,r}, \lambda) \oplus \langle c_{\sigma_{s,r}} \rangle.
		\]
	\end{prop}

	\begin{proof}
		We write every linear relation given in \Cref{lem:linear_action} as
		\[
			z M_{\delta}
			= z + \sum_{i=1}^{\ell(\delta)} \beta_{\delta}^i z^\tr w_{\delta}^i, \quad
			\delta \in \cC,
		\]
		for some $1 \leq \ell(\delta) \leq 3$, $\beta_{\delta}^i \in \QQ^\cA$ and $w_{\delta}^i \in \langle v_0, v_1, \dotsc, v_{s+r}, v_A + v_B \rangle$ for $1 \leq i \leq \ell(\delta)$.
		Observe that since $\bar{h}$ is a positive row integer vector $\beta_{\delta}^i \bar{h}^\tr \neq 0$ for each $1 \leq i \leq \ell(\delta)$.
		From $\lambda \in \cS_{\bar{h}, \bar{m}}(\sigma_{s,r})$ we have $t(\lambda) \bar{h} + \bar{m} \in \Es(\sigma_{s,r}, \lambda) \oplus \langle c_{\sigma_{s,r}} \rangle$ and from $\lambda \in \cS_{\bar{h}M_{\delta}, \bar{m}M_{\delta}}(\sigma_{s,r})$
		\[
			t_\delta(\lambda) \bar{h} M_\delta + \bar{m} M_\delta
			= t(\lambda) \bar{h} + \bar{m} + (t_\delta(\lambda) - t(\lambda)) \bar{h} + \sum_{i=1}^{\ell(\delta)} \beta_{\delta}^i(t_\delta(\lambda) \bar{h} + \bar{m})^\tr w_{\delta}^i
		\]
		belongs to $\Es(\sigma_{s,r}, \lambda) \oplus \langle c_{\sigma_{s,r}} \rangle$, where $t_\delta(\lambda) = - \langle \bar{m} M_\delta, \lambda^\tr \rangle / \langle \bar{h} M_\delta, \lambda^\tr\rangle$.
		We have $\beta_{\delta}^i(t_\delta(\lambda) \bar{h} + \bar{m})^\tr \neq 0$ for each $1 \leq i \leq \ell(\delta)$, since otherwise $t_\delta(\lambda) = - \beta_{\delta}^i \bar{m}^\tr / \beta_{\delta}^i \bar{h}^\tr \in \QQ$ and from \Cref{lem:non_trivial} we would have that $(\bar{h}, \bar{m})$ is a trivial pair for $\sigma_{s,r}$.
		From this and \Cref{lem:linear_action} there exists a sequence $(u_\alpha \colon \alpha \in \cA)$ of row vectors in $\Es(\sigma_{s,r}, \lambda) \oplus \langle c_{\sigma_{s,r}} \rangle$ such that
		\begin{align*}
			v_0 &\in \langle \bar{h}, t(\lambda) \bar{h} + \bar{m}, u_0\rangle \\
			v_{s+r} &\in \langle \bar{h}, t(\lambda) \bar{h} + \bar{m}, u_{s+r}\rangle \\
			v_k &\in \langle \bar{h}, t(\lambda) \bar{h} + \bar{m}, u_{s+r}, u_k\rangle
			&&\quad 1 \leq k \leq s+r-1 \\
			v_A + v_B &\in \langle \bar{h}, t(\lambda) \bar{h} + \bar{m}, u_0, u_1, u_{s+r}, u_A, u_B\rangle
			&&\quad s \neq 0 \\
			v_A + v_B &\in \langle \bar{h}, t(\lambda) \bar{h} + \bar{m}, u_1, u_r, u_{s+r}, u_A, u_B\rangle
			&&\quad s = 0, r \neq 0.
		\end{align*}
		Then the fact that $\bar{h}$ does not belong to $\langle \lambda \rangle^\ort$ finishes the proof.
	\end{proof}

	The proof of the next lemma can be found in \cite[Section 5.2]{Arb18}.

	\begin{lem} \label{lem:linear_dependence}
		For $s, r \ge 0$ let $v_0, v_1, \dotsc, v_{s+r}, v_A, v_B$ be the rows of the matrix $\Omega_{\sigma_{s,r}}$.
		\begin{itemize}
			\item If $s$ and $r$ are odd, then $v_0, v_1, \dotsc, v_{s+r}, v_A$ are linearly independent and
			\[
				v_B
				= -v_A + 2v_s - 2 \sum_{i=1}^{s-1} (-1)^i v_i.
			\]
			\item If $s$ is odd and $r$ is even, then $v_0, v_1, \dotsc, v_{s+r}$ are linearly independent and
			\begin{align*}
				v_A
				&= -v_0 +2v_1 - 2\sum_{i=1}^{s-1} (-1)^i v_{i+1} + \sum_{i=1}^r (-1)^i v_{s+i}, \\
				v_B
				&= v_0 - \sum_{i=1}^r (-1)^i v_{s+i}.
			\end{align*}
			\item If $s$ is even and $r$ is odd, then $v_0, v_1, \dotsc, v_{s+r}, v_A, v_B$ are linearly independent.
			\item If $s$ and $r$ are even, then $v_0, v_1, \dotsc, v_{s+r}, v_A$ are linearly independent and 
			\[
				v_B
				= v_0 - \sum_{i=1}^r (-1)^i v_{s+i}.
			\]
		\end{itemize}
		In particular, we have
		\[
			\dim \langle v_0, v_1, \dotsc, v_{s+r}, v_A + v_B \rangle
			= \begin{cases}
			s+r+1 & \text{if $s$ is odd} \\
			s+r+2 & \text{if $s$ is even}
			\end{cases}
		\]
		and
		\[
			\dim H(\sigma_{s+r})
			= \begin{cases}
			s+r+1 & \text{if $s$ is odd and $r$ is even} \\
			s+r+2 & \text{if $s$ and $r$ are odd, or if $s$ and $r$ are even} \\
			s+r+3 & \text{if $s$ is even and $r$ is odd.} \\
			\end{cases}
		\]
	\end{lem}

	\smallbreak

	We now finish the proof of \Cref{thm:main}.
	From \Cref{lem:lyapunov}, \Cref{lem:contracting} and \Cref{prop:inclusion} we deduce that there exists $\lambda \in \cS$ such that $\dim \Es_{\bar{\D}^\R}(\sigma_{s,r}, \lambda) \leq s+r+3 - g_{\sigma_{s,r}}$ and
	\[
		\langle v_0, v_1, \dotsc, v_{s+r}, v_A + v_B \rangle \cap \langle \lambda \rangle^\ort
		\subseteq \Es_{\bar{\D}^\R}(\sigma_{s,r}, \lambda) \oplus \langle c_{\sigma_{s,r}} \rangle.
	\]
	Therefore
	\[
		\dim \langle v_0, v_1, \dotsc, v_{s+r}, v_A + v_B \rangle - 1
		\leq s+r+3 - g_{\sigma_{s,r}} + 1.
	\]
	\begin{itemize}
		\item If $s$ and $r$ are odd, from \Cref{lem:linear_dependence} the above inequality reads $s+r \leq 6$, but if the genus $g = (s+r+2)/2$ is at least $5$ we obtain a contradiction.
		\item If $s$ is odd and $r$ is even, from \Cref{lem:linear_dependence} the above inequality reads $s+r \leq 7$, but if the genus $g = (s+r+1)/2$ is at least $5$ we obtain a contradiction.
		\item If $s$ is even and $r$ is odd, from \Cref{lem:linear_dependence} the above inequality reads $s+r \leq 7$, but if the genus $g = (s+r+1)/2$ is at least $5$ we obtain a contradiction.
		\item If $s$ and $r$ are even, from \Cref{lem:linear_dependence} the above inequality reads $s+r \leq 4$, but if the genus $g = (s+r)/2$ is at least $3$ we obtain a contradiction.
	\end{itemize}

	\bigbreak

	\textbf{Acknowledgements:} We are grateful to Fabien Durand, Rodolfo Gutiérrez--Romo, Erwan Lanneau and Arnaldo Nogueira for several fruitful discussions.
	The first is a postdoctoral researcher supported by the ANR Program ANR Project IZES ANR-22-CE40-0011.
    The second author is a postdoctoral researcher supported by the Fonds de la Recherche Scientifique -- FNRS.
	The third author acknowledges the ANID Basal Grant FB210005 Center for Mathematical Modeling and Grant ICN2021-044 from the ANID Millennium Science Initiative.

	\sloppy\printbibliography

\end{document}